\documentclass[11pt, a4paper]{amsart}

\usepackage[margin=2.54cm]{geometry}
\usepackage{fancyhdr}

\usepackage{amsmath, amssymb, amsthm}
\usepackage{array}
\usepackage{latexsym}
\usepackage{nicematrix}
\usepackage{bigstrut}
\usepackage{enumerate}
\usepackage{paralist}

\usepackage{mathrsfs}
\usepackage{float}   %choose the emplacement of figures
\usepackage{subcaption}  %for subfigures
\usepackage{pgf,tikz}
\usetikzlibrary{decorations.pathreplacing,calligraphy}
\usepackage{graphicx}
\usepackage[all]{xy}
\usepackage{url}
\usepackage[colorlinks=true,citecolor=cyan,backref=page]{hyperref}
\usepackage[capitalise,nameinlink]{cleveref}

\usepackage{fdsymbol}
\usepackage{todonotes}
\setuptodonotes{size=\tiny}

\makeatletter
\define@key{todonotes}{AS}[]{%
    \setkeys{todonotes}{color=blue!50}}%
\define@key{todonotes}{NCL}[]{%
    \setkeys{todonotes}{color=red!50}}%
\define@key{todonotes}{AA}[]{%
    \setkeys{todonotes}{color=orange!50}}%   
\makeatother

\renewcommand{\P}{\textbf{P}}
\newcommand{\C}{\mathscr{C}}
\newcommand{\T}{\textbf{T}}
\newcommand{\U}{\textbf{U}}
\newcommand{\D}{\textbf{D}}
\newcommand{\I}{\textbf{I}}
\newcommand{\Z}{\textbf{Z}}
\newcommand{\X}{\textbf{X}}

\newcommand{\CC}{\textbf{C}}
\newcommand{\TT}{\mathcal{T}}

\makeatletter
\renewcommand*\env@matrix[1][*\c@MaxMatrixCols c]{%
  \hskip -\arraycolsep
  \let\@ifnextchar\new@ifnextchar
  \array{#1}}
\makeatother

\newtheorem{theorem}{Theorem}[section]
\newtheorem{corollary}[theorem]{Corollary}
\newtheorem{proposition}[theorem]{Proposition}
\newtheorem{lemma}[theorem]{Lemma}

\theoremstyle{definition}
\newtheorem{definition}[theorem]{Definition}
\newtheorem{remark}[theorem]{Remark}
\newtheorem{example}[theorem]{Example}

\newtheorem{question}[theorem]{Question}

\title{Dimension of unicycle posets}

\author[A.~Abram]{Antoine~Abram}
\address[Antoine Abram]{
Mathématiques dicrètes\\
Bâtiment B37\\
Quartier Polytech 1\\
Allée de la Découverte 12\\
4000\\ Liège\\
Belgique}
\email{abram.antoine@uliege.be}
\urladdr{\url{https://antoineabram.codeberg.page/}}

\author[A.~Segovia]{Adrien~Segovia}
\address[Adrien Segovia]{
Laboratoire d’Algèbre, de Combinatoire et d’Informatique Mathématique (LACIM)\\
Universit\'e du Qu\'ebec \`a Montr\'eal\\
CP 8888 Succ. Centre-Ville\\
Montr\'eal, Qu\'ebec, H3C 3P8\\ Canada}
\email{segovia.adrien@courrier.uqam.ca}
\urladdr{\url{https://sites.google.com/view/adriensegovia}}

\begin{document}

\begin{abstract}
Motivated by the study of the dimension of random posets, it was conjectured by Bollob\'as and Brightwell in 1997
that if $P$ is a finite poset whose cover graph contains at most one cycle then its order dimension is at most $3$. In this paper we prove this conjecture by giving a constructive proof with explicit triplets of linear extensions realizing such posets.
\end{abstract}

\maketitle

\hypersetup{linkcolor=black}
\tableofcontents

\section{Introduction}

Let $(\mathbb{R}^n,\leq)$ be the poset defined by $(x_1,\dots,x_n)\leq (y_1,\dots,y_n)$ if $x_i\leq y_i$ for every $i\in [n]$. The (order) \emph{dimension} of a poset $\P$ is the smallest integer $n$ such that $\P$ is isomorphic to a subposet of $(\mathbb{R}^n,\leq)$. Equivalently, this is the minimum number of linear extensions whose intersection is $\P$, where the intersection is the poset on the set $\P$ whose relations are $x\leq y$ if in each of the chosen linear extensions we have $x\leq y$. A list of linear extensions whose intersection is $\P$ is called a \emph{realizer} of $\P$.
Thus $\dim(\P)$ is the smallest size of a realizer of $\P$. This notion was first introduced in \cite{DushnikMiller} and since then was extensively studied (see \cite{trotter2002combinatorics} for a survey).

The \emph{cover graph} of a poset is the graph formed by its undirected Hasse diagram.
A poset is a \emph{tree poset}, a \emph{cycle poset}, or a \emph{unicycle poset}, if their cover graphs are respectively a tree, a cycle, or is connected and contains a unique cycle.

The only posets of dimension $1$ are the chains.
The \textit{planar} lattices, which correspond to the bounded planar posets, are of dimension $2$ \cite{BakerDim2}.
But there exist planar posets of any dimension \cite{KELLY1981135}.
It is known that the tree posets and all planar posets with a minimum element are of dimension at most $3$ \cite{TROTTERplanarTrees}.
In general, knowing the dimension of a given poset is difficult; it is $NP$-complete to determine if a given poset $\P$ is of dimension $n$, for any $n\geq 3$ \cite{YannakakisNPcomplet}, even if $\P$ is of height $2$ \cite{FelsnerNPhauteur2}.

Shortly after the death of Paul Erd\H{o}s in $1996$, two volumes were published with papers by experts on the mathematics Erd\H{o}s was interested in \cite{GrahamNesetril1997VolI,GrahamNesetril1997VolII}. Our main result is a proof of a conjecture of Bollob\'as and Brightwell that appears in one of these papers.

\begin{theorem}[{\cite[Conjecture 1.1]{bollobas1997dimension}}] \label{thmprincipal}
The dimension of a unicycle poset is at most $3$.   
\end{theorem}

This conjecture is motivated by the study of the dimension of \textit{random posets}. Let $[n]:=\{1,\dots,n\}$ with its usual order denoted by $\preccurlyeq$.
A random simple graph $G(n,p)$ on vertices $[n]$ is generated by creating an edge $\{i,j\}$ with probability $p_n$, for any $i\neq j$ independently.
In this model, the probability $p_n$ depends on $n$.
We orient the edges of $G(n,p)$: any edge $\{i,j\}$ with $i\preccurlyeq j$ is directed from $i$ to $j$. A \emph{random poset} $(\P(n,p),\leq)$ is obtained from the directed version of $G(n,p)$ by taking its transitive closure.
It is well-known for random graphs \cite[Theorem 5e]{ErdosRenyi1960}, and pointed out in \cite{bollobas1997dimension}, that when there exists $0\leq c<1$ such that $\lim(p(n) \,n)=c$, then as $n$ goes to infinity $\P(n,p)$ is almost surely a disjoint union of tree posets and unicycle posets. Thus Theorem \ref{thmprincipal} implies that in this case we have almost surely $\dim(\P(n,p))\leq 3$.

We now list a few results that gave upper bounds on the dimension of unicycle posets.
First, from a special case of a result of Trotter \cite{TROTTERproblems}, as mentioned in \cite{bollobas1997dimension}, the dimension of a unicycle poset is at most 5.
In \cite{felsner2015dimension}, it is proved that if a poset has an \emph{outerplanar} cover graph, which means that its cover graph is planar and all its vertices belong to its \mbox{unbounded} face, then its dimension is at most $4$.
Since a unicycle poset has an \mbox{outerplanar} cover graph, it follows that it is of dimension at most $4$. 
More recently, Gao and Kumar gave in \cite{gao2025} a short proof, provided by Brightwell via personal communication, of the fact that the dimension of a unicycle poset is at most 4.

We start with some basics in Section \ref{notation}, and state in Section \ref{sectionUnicycle} that a unicycle poset is a cycle poset with \textit{grafted trees} on each of its vertices (Proposition \ref{prop:unicycleCrowns}). 
We give in Section \ref{sectionCrowns} a specific realizer of the \textit{crowns}, which are the cycle posets with only minimal and maximal elements. In Section \ref{sectionTrees} we give a specific realizer for the tree posets. In particular, we recover the result that the tree posets are of dimension at most $3$.
From these, we deduce realizers for a cycle poset with grafted trees on its vertices that are neither minimal nor maximal elements (Section \ref{sectionTreesonCovers}), and realizers for a crown with grafted trees on any of its vertices (Section \ref{sectionTreesonVertices}).
We finally combine the graftings in order to build a realizer of any unicycle poset in Section \ref{sectionFinale}. These realizers have size $3$, proving Theorem \ref{thmprincipal}\footnote{One can find SageMath code based on the results of this article here: \url{https://github.com/AntoineAbram/Dimension-of-Unicycle-Posets}}.

\section{Acknowledgments} 

We are very thankful to Nathan Chapelier-Laget for numerous discussions on this project. We thank Samuele Giraudo and Hugh Thomas for their careful reading and proof-checking. We thank Clément Chenevière for earlier discussions.
Finally, we want to thank Oliver Pechenik for presenting the conjecture at the BIRS workshop on Lattice theory at Banff in January 2025, and Jane Gao for providing the reference of the conjecture.

\section{Basics}\label{notation}

We denote $[n]:=\{1,2,\dots,n\}$. We denote by $\P^{op}$ the dual poset of $\P$. A poset is \emph{planar} if there exists a planar drawing without crossings of its Hasse diagram, where edges go upward from smaller to bigger elements. A \emph{linear order} of a poset is any total order of the elements of the poset.
A \emph{linear extension} of a poset $(\P, \leq)$ is a linear order $ \preccurlyeq$ on $\P$ such that for all relations $a\leq b$, we have $a \preccurlyeq b$.
For $x,z\in \P$, a cover relation of $\P$ is denoted by $x\lessdot z$. A \emph{saturated chain} is a chain $x_1\lessdot \dots \lessdot x_k$ made of cover relations.
We mostly use lowercase letters for elements, capital letters for linear extensions and bold capital letters for posets or subsets of a poset (\textit{i.e.} $U$ is a linear extension of $\U$).
To avoid using unnecessary symbols, we denote the linear orders and linear extensions as words; for example we write $a\,c\,b\,d$ for $a<c<b<d$. Thus we will use the terminology of words, for example $a\,b\,d$ is a subword of $a\,c\,b\,d$, and $a$ is to the left of $b$ in $a\,c\,b\,d$.
We use the notations $\prod_{i=1}^k a_i:=a_1a_2\cdots a_k$ and $\prod_{i=k}^1 a_i:=a_k a_{k-1}\cdots a_1$.
For two elements, $a\sim b$ means $a\leq b$ or $a\geq b$, and if it is not the case we write $a\not\sim b$ and say that $a$ and $b$ are \emph{incomparable}.
We also use this notation for two sets $\mathbf{X}$ and $\mathbf{Y}$: by writing $\mathbf{X}\not\sim \mathbf{Y}$ we mean that every element of $\mathbf{X}$ is incomparable to every element of $\mathbf{Y}$.
% We also use $\mathbf{X} < \mathbf{Y}$ to say that every element of $\mathbf{X}$ is smaller than every element of $\mathbf{Y}$.

\begin{figure}
    \begin{subfigure}{.35\textwidth}
        \centering  
\begin{tikzpicture}
        \begin{scope}[scale=0.35]
            \draw (-3,-5) -- (3,5) -- (-3,5) -- (3,-5) -- (-3,-5);
            \draw (1,1.67) -- (4,1.67) -- (3,5);
            \draw (-1,-1.67) -- (-4,-1.67) -- (-3,-5);
            \draw (2.7,2.65) node{$\U^-_a$};
            \draw (0,0) node{$\bullet$};
            \draw (0.2,0) node[right]{$a$};
            \draw (0,3) node{$\U^+_a$};
            \draw (0,-3) node{$\D^-_a$};
            \draw (-2.7,-2.65) node{$\D^+_a$};
            \draw [decorate,decoration={brace,amplitude=7pt,mirror,raise=2ex}]
  (3.5,0.6) -- (3.5,5) node[midway,xshift=2.5em,yshift=0em]{$\U_a$};
  \draw [decorate,decoration={brace,amplitude=7pt, mirror,raise=2ex}]
  (-3.5,-0.6) -- (-3.5,-5) node[midway,xshift=-2.5em,yshift=0em]{$\D_a$};
        \end{scope}    
    \end{tikzpicture}
        \caption{A tree poset $\T$ with $a\in \T$. 
        \label{fig:simpletree}}        
    \end{subfigure}
    ~
    \begin{subfigure}{.6\textwidth}
        \centering
        \begin{tikzpicture}[scale=0.516]
            \draw (0,0)--(1,2)--(0,4)--(-1,2)--(0,0); 
            \draw (0,0) node [below]{$x$};
            \draw (0,4) node [above]{$z$};
            \draw (1,2) node [right]{$b$};
            \draw (-1,2) node [left]{$a$};
            \draw (0,0) node {$\bullet$};
            \draw (0,4) node {$\bullet$};
            \draw (1,2) node {$\bullet$};
            \draw (-1,2) node {$\bullet$};
\begin{scope}[xshift= 3cm,scale=1]
\draw (8.25,1)--(9,0)--(9,4)--(12,0)--(12,4)--(0,0)--(0,4)--(3,0)--(3,4)--(6,0)--(6,4)--(6.75,3);
            \draw[dotted, thick] (8.25,1)--(6.75,3);
            
            \draw (7.3,-0.1) node [below]{\footnotesize$\cdots$};
            \draw (7.3,5) node [below]{\footnotesize$\cdots$};
            
            \draw (0,0) node [below]{$x_1$};
            \draw (3,0) node [below]{$x_2$};
            \draw (6,0) node [below]{$x_3$};
            \draw (9,0) node [below]{$x_{n-1}$};
            \draw (12,0) node [below]{$x_n$};
            \draw (0,4) node [above]{$z_1$};
            \draw (3,4) node [above]{$z_2$};
            \draw (6,4) node [above]{$z_3$};
            \draw (9,4) node [above]{$z_{n-1}$};
            \draw (12,4) node [above]{$z_n$};
            \draw (0,0) node {$\bullet$};
            \draw (3,0) node {$\bullet$};
            \draw (6,0) node {$\bullet$};
            \draw (9,0) node {$\bullet$};
            \draw (12,0) node {$\bullet$};
            \draw (0,4) node {$\bullet$};
            \draw (3,4) node {$\bullet$};
            \draw (6,4) node {$\bullet$};
            \draw (9,4) node {$\bullet$};
            \draw (12,4) node {$\bullet$};
            
            \draw (0,3) node [left]{\textcolor{orange}{\scriptsize 1}};
            \draw (0.95,3) node [right]{\textcolor{orange}{\scriptsize 2}};
            \draw (3,3) node [left]{\textcolor{orange}{\scriptsize 3}};
            \draw (3.95,3) node [right]{\textcolor{orange}{\scriptsize 4}};
            \draw (6,3) node [left]{\textcolor{orange}{\scriptsize 5}};
            \draw (8.9,1) node [right]{\textcolor{orange}{\scriptsize$2n-3$}};
            \draw (9.9,2.8) node [right]{\textcolor{orange}{\scriptsize$2n-2$}};
            \draw (12,1) node [right]{\textcolor{orange}{\scriptsize$2n-1$}};
            \draw (8,3.1) node {\textcolor{orange}{\scriptsize$2n$}};
\end{scope}
        \end{tikzpicture}
        \caption{The crowns $\C_1$ and $\C_n$ for $n\geq 2$.
        The orange edge labelling is the one explained in Definition \ref{def:crown}.\label{fig:crown}}
    \end{subfigure}
    \caption{\label{fig:treeandcrown}}
\end{figure}

Let $\T$ be a tree poset and $a\in \T$. See Figure \ref{fig:simpletree} and Example \ref{ex:treeslinear} for the way we draw the trees and illustrations of the following definitions.
For any $x\in \T\setminus \{a\}$ there exists a unique non-empty path in the Hasse-diagram of $\T$ from $a$ to $x$.
Denote $\U_a$ the set of elements $x\in \T\setminus \{a\}$ such that the path from $a$ to $x$ starts with an up-step, meaning an element bigger than $a$, and similarly denote $\D_a$ the set of elements that can be reached from $a$ by starting with a down-step.
We have $\T\setminus\{a\} = \U_a \sqcup \D_a$.
Denote by $\U_a^{+} \subseteq \U_a$ the subset of the elements of $\U_a$ that are bigger than $a$, and $\U_a^{-}:=\U_a\setminus (\U_a^{+})$ the complement of $\U_a^+$ in $\U_a$.
Similarly, denote by $\D_a^{-} \subseteq \D_a$ the subset of the elements of $\D_a$ that are smaller than $a$, and $\D_a^{+}:=\D_a\setminus (\D_a^{-})$ the complement of $\D_a^-$ in $\D_a$.
We have $\U_a = \U_a^{+} \sqcup \U_a^{-}$ and $\D_a = \D_a^{-} \sqcup \D_a^{+}$.
It will be useful to let $\I_a^+:= \U_a \sqcup \D_a^+$ and $\I_a^-:= \D_a \sqcup \U_a^-$.
Observe that in $\T$ all the elements incomparable to $a$ are those in $\U_a^{-}\sqcup \D_a^{+}$. When no confusion is possible, we drop the subscript $a$ from the notations that we just introduced.

For a linear extension $L$ of $\P$ represented as a word, we denote by $L^{rev}$ the reverse of $L$, meaning the word obtained from $L$ by reading it from right to left. The following result is immediate:

\begin{lemma} \label{lem:realizerPop}
    Let $L_1,\dots , L_k$ be a realizer of a poset $\P$. Then $L_1^{rev},\dots , L_k^{rev}$ is a realizer of $\P^{op}$.
\end{lemma}

In most of the proofs we will make extensive use of the following easy result without reference to it:

\begin{lemma}\label{lem:inducedlinext}
Let $\P=\mathbf{X} \cup \mathbf{Y}$ and let $L_1,\dots,L_k$ be linear extensions of $\P$ such that the restrictions of these linear extensions to $\mathbf{X}$, respectively $\mathbf{Y}$, give a realizer of $\mathbf{X}$, rescpectively $\mathbf{Y}$. If for all $(x,y)\in \mathbf{X} \times \mathbf{Y}$ such that $x\not\sim y$, there exist $i,j\in [k]$ such that $x$ is to the left of $y$ in $L_i$ and to the right of $y$ in $L_j$, then $L_1,\dots,L_k$ is a realizer of $\P$. 
\end{lemma}

% The following result is immediate:

% \begin{lemma} \label{lem:realizerPop}
%     Let $L_1,\dots , L_k$ be a realizer of a poset $\P$.
%     For a linear extension $L$ of $\P$ represented as a word, denote by $L^{rev}$ the reverse of $L$, meaning the word obtained from $L$ by reading it from right to left.
%     Then $L_1^{rev},\dots , L_k^{rev}$ is a realizer of $\P^{op}$.
% \end{lemma}

% In most of the proofs we make extensive use of the following easy result without reference to it:

% \begin{lemma}\label{lem:inducedlinext}
%     Let $\P$ be a poset and $\P'$ be a subposet of $\P$.
%     Suppose that $L_1,\dots,L_k$ is a realizer of $\P$.
%     For all $i$, denote by $L_i'$ the subword of $L_i$ on letters in $\P'$.
%     Then $L_1',\dots,L_k'$ is a realizer of $\P'$.
% \end{lemma}

% It follows that the dimension of a subposet $\P'$ of $\P$ is always less or equal to that of $\P$. 

\begin{figure}
    \centering
\begin{tikzpicture}[scale=0.55]
\draw[red,thick] (0,0)--(0,8)--(10,2)--(10,7)--(0,0);
\draw (-2,2)--(-1,1)--(0,2);
\draw (-2,4)--(-1,3)--(0,5)--(-1,7)--(-2,6);
\draw (-1,10)--(0,8)--(1,10);
\draw (7.2,5)--(5,8)--(4,10);
\draw (4,7)--(5,8)--(6,10);
\draw (8,1)--(10,0)--(10,2);
\draw (8,7)--(9,9)--(10,7)--(11,9)--(12,7);

\draw[red,thick] (0,0) node {$\bullet$}; 
\draw (0,2) node {$\bullet$};
\draw (0,5) node {$\bullet$};
\draw[red,thick] (0,8) node {$\bullet$};
\draw[red,thick] (0,0) node [left]{$x_1$}; 
\draw (0,2) node [right]{$y_1^1$};
\draw (0,5) node [right]{$y_2^1$};
\draw[red,thick] (0.1,8) node [right]{$z_1$};

\draw (-1,1) node {$\bullet$};
\draw (-1,3) node {$\bullet$};
\draw (-1,7) node {$\bullet$};
\draw (-1,10) node {$\bullet$};
\draw (-1,1) node [below]{$a$};
\draw (-1,3) node [below]{$e$};
\draw (-1,7) node [above]{$h$};
\draw (-1,10) node [left]{$m$};

\draw (-2,2) node {$\bullet$};
\draw (-2,4) node {$\bullet$};
\draw (-2,6) node {$\bullet$};
\draw (-2,2) node [above]{$b$};
\draw (-2,4) node [above]{$f$};
\draw (-2,6) node [below]{$g$};

\draw (1,10) node {$\bullet$};
\draw (1,10) node [right]{$n$};

\draw (3.3,6) node {$\bullet$};
\draw (3.3,6) node [below left]{$y_1^2$};

\draw (4,7) node {$\bullet$};
\draw (4,10) node {$\bullet$};
\draw (4,7) node [left]{$i$};
\draw (4,10) node [left]{$k$};

\draw (5,8) node {$\bullet$};
\draw (5.1,8) node [right]{$j$};

\draw (6,10) node {$\bullet$};
\draw (6,10) node [right]{$l$};

\draw (7.2,5) node {$\bullet$};
\draw (7.2,5.3) node [below right]{$y_1^4$};

\draw (8,1) node {$\bullet$};
\draw (8,7) node {$\bullet$};
\draw (8,1) node [left]{$d$};
\draw (8,7) node [left]{$o$};

\draw (9,9) node {$\bullet$};
\draw (9,9) node [above]{$q$};

\draw (10,0) node {$\bullet$};
\draw[red,thick] (10,2) node {$\bullet$};
\draw (10,4.5) node {$\bullet$};
\draw[red,thick] (10,7) node {$\bullet$};
\draw (10,0) node [right]{$c$};
\draw[red,thick] (10,2) node [right]{$x_2$};
\draw (10,4.5) node [right]{$y_1^3$};
\draw[red,thick] (10,7) node [right]{$z_2$};

\draw (11,9) node {$\bullet$};
\draw (12,7) node {$\bullet$};
\draw (11,9) node [above]{$r$};
\draw (12,7) node [below]{$p$};
\end{tikzpicture}
    \caption{A unicycle poset $\P$ with an underlying cycle poset $\CC_2$ on the elements $x_1,\,y_1^1,y_2^1,\,z_1,\,y_1^2,\,x_2,\,y_1^3,\,z_2,\,y_1^4$. The underlying crown $\C_2$ is depicted in red.}
    \label{fig:runningexunicycle}
\end{figure}

\begin{example}
For the unicycle poset $\P$ of Figure \ref{fig:runningexunicycle}, the following three linear orders form a realizer of $\P$ (they are obtained using Proposition \ref{conjFinal}):
\begin{align*}
\mathcal{L}_1=  g\; e\; f\; a\; b\; x_1\; y_1^1\; y_2^1\; h\; c\; x_2\; y_1^2\; z_1\; m\; n\; y_1^3\; y_1^4\; z_2\; o\; q\; p\; r\; i\; j\; k\; l\; d\\
\mathcal{L}_2=o\; p\; c\; d\; x_2\; y_1^3\; y_1^2\; i\; x_1\; y_1^4\; j\; k\; l\; z_2\; r\; q\; a\; y^1_1\; e\; y_2^1\; z_1\; n\; m\; g\; h\; f\; b\\
\mathcal{L}_3=x_1\; y_1^4\; i\; j\; l\; k\; c\; d\; x_2\; a\; b\; y^1_1\; e\; f\; y_2^1\; g\; h\; y_1^3\; z_2\; p\; r\; o\; q\; y_1^2\; z_1\; n\; m
\end{align*}

For example, the only elements below $y_1^2$ are $c$ and $x_2$. This can be seen by intersecting the three linear orders $\mathcal{L}_1$, $\mathcal{L}_2$ and $\mathcal{L}_3$. Indeed, $c$ and $x_2$ are to the left of $y_1^2$ in each of the linear orders, whereas the only other elements to the left of $y_1^2$ in $\mathcal{L}_2$ are $o,\,p,\,d,\,y_1^3$, but these elements are to the right of $y_1^2$ in $\mathcal{L}_1$.
\end{example}

\begin{remark}
We want to point out that after finding a specific realizer of the crowns in Lemma \ref{lem:ExtLinCouronne} and of the tree posets in Proposition \ref{prop:Hardcore}, all the remaining results of the paper starting from Corollary \ref{cor:extlinarbre} will have always the same kind of proofs. We will be given three linear orders of a poset $\P$ and the fact that these are linear extensions is straightforward to prove (the pictures help). Then we will observe that realizers of some relevant subposets are obtained using previous results as subwords of the three given linear extensions. We will say that we \emph{recover} these subposets. Then we will finish the proofs using Lemma \ref{lem:inducedlinext}, by proving the incomparabilities that are not part of the subposets described previously.
\end{remark}

\section{Unicycle posets}\label{sectionUnicycle}

A \emph{cycle poset} is a poset $\P$ whose cover graph is a cycle. A cycle poset such that all its elements are minimal or maximal elements is called a \emph{fundamental cycle poset}.
An immediate result is that a cycle poset has all its vertices of degree $2$ and the same number of minimal and maximal elements.
It follows that the fundamental cycle posets are the crowns $\C_n$ for $n\geq 2$, whose definition is the following (see Figure \ref{fig:crown}):

\begin{definition}\label{def:crown}
    For $n\geq 2$, the \emph{crown} $\C_{n}$ is the poset having $2n$ elements $x_1,x_2,\dots,x_n ,z_1,z_2,\dots,z_n$ where for all $i$ the $x_i$'s are minimal elements and the $z_i$'s are maximal elements and the relations are $x_{i+1}$ is less than $z_i$ and $z_{i+1}$ for indices taken cyclically.
    We also define a numbering of all covers of the crown $\C_n$ for $n\geq 2$; for all $i\in [n]$, the number $2i-1$ is assigned to the cover $x_i \lessdot z_i$, and $2i$ is assigned to $x_{i+1} \lessdot z_i$.
It means that for all $p\in [2n]$ the cover numbered by $p$ is $x_{\lfloor \frac{p}{2}\rfloor +1} \lessdot z_{\lceil \frac{p}{2}\rceil}$.
We define the crown $\C_1$ in a different way as the square poset on four elements, with minimum $x$, maximum $z$ and incomparable elements $a,b$.
\end{definition}

In the sequel, the indices within the crowns will always be taken cyclically. 

As defined in \cite{Spinrad88}, an (edge) \emph{subdivision} $P'$ of a poset $P$ is a poset obtained by adding elements ``on'' the edges of $P$ in its Hasse diagram. Two subdivisions of the crown $\C_2$ are represented in \cref{fig:Sub}. The following is immediat.

\begin{lemma}\label{lemma:cycleiscrown}
    % A cycle poset $\CC$ is a crown with a certain number of added vertices (possibly zero) on each cover of the crown. 
A cycle poset $\CC$ is a subdivision of a crown.
\end{lemma}
% \ant{Do we need this proof? I feel it is trivial and the meet of the proof is not given, why the poset has the same number of maximal and minimal elements...}
% \begin{proof}
%     If $\CC$ has a minimum element (thus a maximum), it is $\C_1$ with possibly extra vertices added on its covers since all its vertices have degree $2$.
%     If $\CC$ does not have a minimum element, consider the subposet of $\CC$ defined by keeping only its minimal and maximal elements. It is a fundamental cycle, thus a crown $\C_n$ for $n\geq 2$.
%     All other elements of $\CC$ are of degree $2$ and since they are not minimal or maximal elements, they form chains between two such elements.
%     The result follows.
% \end{proof}

We denote by $\CC_n$ a cycle poset that has $\C_n$ as a subposet; we say that $\C_n$ is the \emph{underlying crown of} $\CC_n$.
A \emph{rooted tree poset} is a pair $(\T,r)$ of a tree poset $\T$ and a distinguished element $r$ of $\T$.

\begin{definition}\label{grafting}
    Let $\P$ be a poset.
    Let $\TT(\P):=\{(\T_x,r_x)\mid x\in\P\}$ be a family of rooted tree posets indexed by $\P$.
    The \emph{grafting of $\TT$ on $\P$} is the poset on $\textbf{R}:=\{(x,t)\mid x\in \P,\, t\in\T_x\}$ such that $(a,b)\leq (c,d)$ if and only if
    \begin{enumerate}[(i)]
        \item $a=c$ and $b\leq d$ in $\T_a$; or
        \item $a<c$ in $\P$, $b\leq r_a$ in $\T_a$ and $r_c\leq d$ in $\T_c$.
    \end{enumerate}
\end{definition}

% The grafting is a special case of a more general construction called \emph{amalgamation} \cite{TROTTERproblems}.
Intuitively, the Hasse diagram of the grafting of $\TT(\P)$ on $\P$ corresponds to the directed graph obtained by identifying each $x\in\P$ with $r_x$. For example, the result of the grafting of rooted trees on the poset of Figure \ref{fig:cycleC1} is given in Figure \ref{fig:unicycleC1}.

\begin{remark} \label{rem:graft}
The grafting defined in this paper is a specific example of a construction
called \emph{amalgamation} \cite{TROTTERproblems}.
% The amalgamation consists in replacing each element with posets that has a pointed element. The grafting is a subdivision of the former poset to which one does an amalgamation with a family of trees.

Note that both the subdivision and the grafting operations on a poset can increase its dimension. For the grafting, it is intuitive. For the subdivision, see \cref{fig:Sub2} which shows a subdivision of dimension 3 of the crown poset $\C_2$ which has dimension 2.
\end{remark}

% \ant{Peut-être que l'on peut utiliser cette remarque pour expliquer pourquoi on différencie les arbres sur les covers de ceux sur les éléments de la couronne.}
\begin{figure}
    \begin{subfigure}{.49\textwidth}
        \centering
        \begin{tikzpicture}
            \draw (0,0)--(0,1)--(0,2)--(1,0)--(1,2)--(0,0);
            \draw (0,0) node{$\bullet$};
            \draw[red] (0,1) node{$\bullet$};
            \draw (0,2) node{$\bullet$};
            \draw (1,0) node{$\bullet$};
            \draw (1,2) node{$\bullet$};
        \end{tikzpicture}
        \caption{A subdivision of $\C_2$ which has dimension 2.}
        \label{fig:sub1}
    \end{subfigure}
    ~
    \begin{subfigure}{.49\textwidth}
        \centering
        \begin{tikzpicture}
            \draw (0,0)--(0,1)--(0,2)--(1,0)--(1,2)--(0,0);
            \draw (0,0) node{$\bullet$};
            \draw[red] (0,1) node{$\bullet$};
            \draw (0,2) node{$\bullet$};
            \draw[red] (0.25,0.5) node{$\bullet$};
            \draw (1,0) node{$\bullet$};
            \draw (1,2) node{$\bullet$};
        \end{tikzpicture}
        \caption{A subdivision of $\C_2$ which has dimension 3.}
        \label{fig:Sub2}
    \end{subfigure}
    \caption{\label{fig:Sub}}
\end{figure}

Since a unicycle poset is a cycle poset with tree posets grafted to any of its vertices, it follows:

\begin{proposition}\label{prop:unicycleCrowns}
    Let $\P$ be a unicycle poset. There exist $n\geq 1$ and $\TT(\CC_n)$ such that the grafting of $\TT$ on $\CC_n$ is isomorphic to $\P$.
\end{proposition}

For $\P$ a unicycle poset, the cycle poset $\CC_n$ from which $\P$ is constructed is called its \emph{underlying cycle poset}. For example see Figure \ref{fig:unicycleC1} for a unicycle poset whose underlying cycle poset is the cycle poset in Figure \ref{fig:cycleC1}.
See Figure \ref{fig:generalunicycle} for an example of a unicycle poset whose underlying cycle poset is $\CC_n$ with $n\geq 2$.
Using Proposition \ref{prop:unicycleCrowns}, a unicycle poset $\P$ can be describe using a pair $(\CC_n,\TT(\CC_n))$.

% As explained in \cref{rem:graft} and \cref{prop:unicycleCrowns}, a unicyclic poset $\P$ is a grafting on a cyclic poset $\CC_n$.
The graftings of trees on extremal elements of $\CC_n$ have a different impact on linear extensions compared to graftings on non-extremal elements.
Thus, we treat them in two different manners.
The extremal graftings (on crown posets) are treated in \cref{sectionTreesonVertices} and the other ones in \cref{sectionTreesonCovers}.
We wrap everything together in \cref{sectionFinale}.

\begin{figure}
    \begin{subfigure}{.49\textwidth}
        \centering
        \begin{tikzpicture}
            \begin{scope}[scale=2.2]
                \draw (0,0)--(1,1)--(0,2)--(-1,1)--(0,0); 
                \draw (0.1,0) node{$x$};
                \draw (0,0) node{$\bullet$};
                \draw (-1.1,1) node{$a$};
                \draw (-1,1) node{$\bullet$};
                \draw (0.1,0) node{$x$};
                \draw (0,0) node{$\bullet$};
                \draw (1.1,1) node{$b$};
                \draw (1,1) node{$\bullet$};
                \draw (0.1,0) node{$x$};
                \draw (0,0) node{$\bullet$};
                \draw (0.1,2) node{$z$};
                \draw (0,2) node{$\bullet$};
                
                \begin{scope}[scale=0.1]
                    \draw (-1,-3)--(1,3)--(-1,3)--(1,-3)--(-1,-3);
                    \draw (1,3)--(1.5,2)--(0.67,2);
                    \draw (-1,-3)--(-1.5,-2)--(-0.67,-2);
                \end{scope}
                
                \begin{scope}[xshift=1cm,yshift=1cm,scale=0.1]
                    \draw (-1,-3)--(1,3)--(-1,3)--(1,-3)--(-1,-3);
                    \draw (1,3)--(1.5,2)--(0.67,2);
                    \draw (-1,-3)--(-1.5,-2)--(-0.67,-2);
                \end{scope}
                
                \begin{scope}[yshift=2cm,scale=0.1]
                    \draw (-1,-3)--(1,3)--(-1,3)--(1,-3)--(-1,-3);
                    \draw (1,3)--(1.5,2)--(0.67,2);
                    \draw (-1,-3)--(-1.5,-2)--(-0.67,-2);
                \end{scope}
                
                \begin{scope}[xshift=-1cm,yshift=1cm,scale=0.1]
                    \draw (-1,-3)--(1,3)--(-1,3)--(1,-3)--(-1,-3);
                    \draw (1,3)--(1.5,2)--(0.67,2);
                    \draw (-1,-3)--(-1.5,-2)--(-0.67,-2);
                \end{scope}
                
                \begin{scope}[xshift=0.4cm,yshift=0.4cm,scale=0.07,rotate=40]
                    \draw (-1,-3)--(1,3)--(-1,3)--(1,-3)--(-1,-3);
                    \draw (1,3)--(1.5,2)--(0.67,2);
                    \draw (-1,-3)--(-1.5,-2)--(-0.67,-2);
                \end{scope}
                
                \begin{scope}[xshift=0.6cm,yshift=0.6cm,scale=0.07,rotate=40]
                    \draw (-1,-3)--(1,3)--(-1,3)--(1,-3)--(-1,-3);
                    \draw (1,3)--(1.5,2)--(0.67,2);
                    \draw (-1,-3)--(-1.5,-2)--(-0.67,-2);
                \end{scope}
                
                \begin{scope}[xshift=0.4cm,yshift=1.6cm,scale=0.07,rotate=-40]
                    \draw (-1,-3)--(1,3)--(-1,3)--(1,-3)--(-1,-3);
                    \draw (1,3)--(1.5,2)--(0.67,2);
                    \draw (-1,-3)--(-1.5,-2)--(-0.67,-2);
                \end{scope}
                
                \begin{scope}[xshift=0.6cm,yshift=1.4cm,scale=0.07,rotate=-40]
                    \draw (-1,-3)--(1,3)--(-1,3)--(1,-3)--(-1,-3);
                    \draw (1,3)--(1.5,2)--(0.67,2);
                    \draw (-1,-3)--(-1.5,-2)--(-0.67,-2);
                \end{scope}
                
                \begin{scope}[xshift=-0.4cm,yshift=0.4cm,scale=0.07,rotate=-40]
                    \draw (-1,-3)--(1,3)--(-1,3)--(1,-3)--(-1,-3);
                    \draw (1,3)--(1.5,2)--(0.67,2);
                    \draw (-1,-3)--(-1.5,-2)--(-0.67,-2);
                \end{scope}
                
                \begin{scope}[xshift=-0.6cm,yshift=0.6cm,scale=0.07,rotate=-40]
                    \draw (-1,-3)--(1,3)--(-1,3)--(1,-3)--(-1,-3);
                    \draw (1,3)--(1.5,2)--(0.67,2);
                    \draw (-1,-3)--(-1.5,-2)--(-0.67,-2);
                \end{scope}
                
                \begin{scope}[xshift=-0.4cm,yshift=1.6cm,scale=0.07,rotate=40]
                    \draw (-1,-3)--(1,3)--(-1,3)--(1,-3)--(-1,-3);
                    \draw (1,3)--(1.5,2)--(0.67,2);
                    \draw (-1,-3)--(-1.5,-2)--(-0.67,-2);
                \end{scope}
                
                \begin{scope}[xshift=-0.6cm,yshift=1.4cm,scale=0.07,rotate=40]
                    \draw (-1,-3)--(1,3)--(-1,3)--(1,-3)--(-1,-3);
                    \draw (1,3)--(1.5,2)--(0.67,2);
                    \draw (-1,-3)--(-1.5,-2)--(-0.67,-2);
                \end{scope}
            \end{scope}
        \end{tikzpicture}
        \caption{A unicycle poset whose underlying cycle poset is the poset in Figure \ref{fig:cycleC1}.}
        \label{fig:unicycleC1}
    \end{subfigure}
    ~
    \begin{subfigure}{.49\textwidth}
        \centering
        \begin{tikzpicture}
            \begin{scope}[scale=2.3]
                \draw (0,0)--(1,1)--(0,2)--(-1,1)--(0,0); 
                \draw (0.1,0) node{$x$};
                \draw (0,0) node{$\bullet$};
                \draw (-1.1,1) node{$a$};
                \draw (-1,1) node{$\bullet$};
                \draw (0.1,0) node{$x$};
                \draw (0,0) node{$\bullet$};
                \draw (1.1,1) node{$b$};
                \draw (1,1) node{$\bullet$};
                \draw (0.1,0) node{$x$};
                \draw (0,0) node{$\bullet$};
                \draw (0.1,2) node{$z$};
                \draw (0,2) node{$\bullet$};
                
                \begin{scope}[xshift=1cm,yshift=1cm,scale=0.1]
                    \draw (0,0) node{$\bullet$};
                \end{scope}
                
                \begin{scope}[xshift=-1cm,yshift=1cm,scale=0.1]
                   \draw (0,0) node{$\bullet$};
                \end{scope}
                
                \begin{scope}[xshift=0.4cm,yshift=0.4cm,scale=0.07]
             \draw (0,0) node{$\bullet$};
                \end{scope}
                
                \begin{scope}[xshift=0.6cm,yshift=0.6cm,scale=0.07]
                   \draw (0,0) node{$\bullet$};
                \end{scope}
                
                \begin{scope}[xshift=0.4cm,yshift=1.6cm,scale=0.07]
                    \draw (0,0) node{$\bullet$};
                \end{scope}
                
                \begin{scope}[xshift=0.6cm,yshift=1.4cm,scale=0.07]
                    \draw (0,0) node{$\bullet$};
                \end{scope}
                
                \begin{scope}[xshift=-0.4cm,yshift=0.4cm,scale=0.07]
                   \draw (0,0) node{$\bullet$};
                \end{scope}
                
                \begin{scope}[xshift=-0.6cm,yshift=0.6cm,scale=0.07]
                    \draw (0,0) node{$\bullet$};
                \end{scope}
                
                \begin{scope}[xshift=-0.4cm,yshift=1.6cm,scale=0.07]
                    \draw (0,0) node{$\bullet$};
                \end{scope}
                
                \begin{scope}[xshift=-0.6cm,yshift=1.4cm,scale=0.07]
                   \draw (0,0) node{$\bullet$};
                \end{scope}
            \end{scope}
        \end{tikzpicture}
        \caption{A cycle poset with underlying crown $\C_1$.}
        \label{fig:cycleC1}
    \end{subfigure}
    \caption{\label{fig:cyclesC1}}
\end{figure}

\section{Crowns}
\label{sectionCrowns}

The fact that the crowns $\C_1$ and $\C_2$ have dimension $2$ and the crown $\C_n$ for $n\geq 3$ has dimension $3$ is a well-known result \cite[Theorem 4.3]{TROTTERcrowns}.
But we need to work with a specific realizer for later results. After a thorough study of the realizers of crowns, we found that the following realizer is well suited for our construction.

\begin{lemma}
\label{lem:ExtLinCouronne}
For all $n\geq 2$, the following three linear orders of $\C_n$ form a realizer:
    \begin{gather}
        x_1 \;  x_2 \; z_1 \prod_{i=2}^{n-1} \big(x_{i+1} \; z_i\big) \; z_n,\label{exten1}\\
        x_2 \; \prod_{i=2}^{n-1} \big(x_{i+1} \; z_i\big) \;x_1 \; z_n \; z_1,\label{exten2}\\
        x_1 \; \prod_{i=n}^{2} \big(x_{i} \; z_i\big) \; z_1\label{exten3}.
    \end{gather}
\end{lemma}

\begin{proof}
    A linear order is a linear extension of $\C_n$ if and only if for any $i$, $z_i$ is to the right of $x_i$ and $x_{i+1}$.
    It follows that the linear orders \eqref{exten1}, \eqref{exten2} and \eqref{exten3} of $\C_n$ are linear extensions of $\C_n$.
    It remains to prove that we get all the incomparabilities by intersecting the linear extensions \eqref{exten1}, \eqref{exten2} and \eqref{exten3}; the only comparabilities being $x_i<z_{i-1},z_i$. 
    For all $i$ the $z_i$'s are mutually incomparable by looking at \eqref{exten1} and \eqref{exten3}.
    For $x_1$, with \eqref{exten1} and \eqref{exten2} we have $x_1<z_1,z_n$ and no other relations. Since we treated $x_1$, we can look at the linear extensions restricted to all elements except $x_1$.
    Let $i\neq 1$. We look at $x_i$. By \eqref{exten1} and \eqref{exten3} we have only the relations $x_i<z_{i-1},z_i$.
    This concludes the proof.
\end{proof}

\begin{lemma}
    A realizer of the crown $\C_1$ is given by the two linear extensions $x\,a\,b\,z$ and $x\,b\,a\,z$.    
\end{lemma}

Even if $\C_2$ is of dimension $2$, we still use the realizer of size three given by Lemma \ref{lem:ExtLinCouronne}.
Similarly, for $\C_1$ we consider the realizer of size 3 obtained by repeating its first linear extension $x\,a\,b\,z$.

\section{Tree posets}
\label{sectionTrees}

We know that the tree posets are of dimension at most $3$ \cite[Theorem 2]{TROTTERplanarTrees} but, as for the crowns, an important step in the proof of our main result is to obtain a specific realizer of the tree posets (see Corollary \ref{cor:extlinarbre}).
Recall the notations from Section \ref{notation} related to the tree posets.
If a tree poset $\T$ has a minimum $a$, choose a planar representation of $\T$, which means choosing an order on the covers of any vertex. 
The \emph{prefix left-to-right reading word} of $\T$ is the word obtained by the following procedure; starting at $a$, each time one visits a vertex $v$, one then has to recursively visit the cover elements of $v$ starting from the leftmost one.
The \emph{prefix right-to-left reading word} of $\T$ is exactly the same procedure but visiting the cover elements from right to left instead.
See Figure \ref{fig:rootedtree} for an example of the following result:

\begin{figure}
    \begin{subfigure}{.49\textwidth}
        \centering  
\begin{tikzpicture}
\draw (-1,1)--(0,0)--(2,2);
\draw (1,1)--(0,2);
\draw (0,0) node{$\bullet$};
\draw (-1,1) node{$\bullet$};
\draw (1,1) node{$\bullet$};
\draw (0,2) node{$\bullet$};
\draw (2,2) node{$\bullet$};
\draw (0,0) node[below]{$a$};
\draw (-1,1) node[left]{$e_1$};
\draw (1,1) node[right]{$e_2$};
\draw (0,2) node[left]{$e_3$};
\draw (2,2) node[right]{$e_4$};  
\end{tikzpicture}
        \caption{A tree poset $\P$ with a minimum $a$. The left cover of $a$ is $e_1$ and its right cover is $e_2$. A realizer of size $2$ is given by $a\,e_1\,e_2\,e_3\,e_4$ and $a\,e_2\,e_4\,e_3\,e_1$.}
        \label{fig:rootedtree}
    \end{subfigure}
    ~
    \begin{subfigure}{.49\textwidth}
        \centering
\begin{tikzpicture}
\begin{scope}[scale=1.2]
 \draw (-1,1)--(0,0)--(2,2);
\draw (1,1)--(0,2);
\draw (0,0) node{$\bullet$};
\draw (-1,1) node{$\bullet$};
\draw (1,1) node{$\bullet$};
\draw (0,2) node{$\bullet$};
\draw (2,2) node{$\bullet$};
\draw (0,0) node[below]{$a$};
\draw (-1,1) node[left]{$e_1$};
\draw (1,1) node[right]{$e_2$};
\draw (0,2) node[left]{$e_3$};
\draw (2,2) node[right]{$e_4$};   
\begin{scope}[rotate around={-10:(-1,1)}]
    \draw (-1,1)--(-0.5,0)--(-1.5,0)--(-1,1);
\draw (-1,0.3) node{$\T_1$};
\end{scope}

\begin{scope}[xshift=3cm, yshift=1cm, rotate around={20:(-1,1)}]
    \draw (-1,1)--(-0.5,0)--(-1.5,0)--(-1,1);
\draw (-1,0.3) node{$\T_4$};
\end{scope}

\begin{scope}[xshift=2cm, rotate around={10:(-1,1)}]
    \draw (-1,1)--(-0.5,0)--(-1.5,0)--(-1,1);
\draw (-1,0.3) node{$\T_2$};
\end{scope}

\begin{scope}[xshift=1cm, yshift=1cm]
    \draw (-1,1)--(-0.5,0)--(-1.5,0)--(-1,1);
\draw (-1,0.3) node{$\T_3$};
\end{scope}
\end{scope}
\end{tikzpicture}      
        \caption{Illustration for the proof of Proposition \ref{prop:Hardcore}. \label{fig:proofPropTree}}
    \end{subfigure}
    \caption{\label{fig:PropTrees}}
\end{figure}

\begin{lemma}\label{lemma1}
    Let $\P$ be a tree poset.
    If $\P$ has a minimum or a maximum, then its dimension is at most $2$.
\end{lemma}

\begin{proof}
    First suppose that $\P$ has a minimum.
    It is easy to see that a realizer of size $2$ of $\P$ is given by the linear extensions obtained from the prefix left-to-right and prefix right-to-left reading words of the tree.
    If $\P$ has a maximum, the result follows by Lemma \ref{lem:realizerPop} as $\P^{op}$ has a minimum. 
\end{proof}

\begin{proposition}\label{prop:Hardcore}
    Let $\P$ be a tree poset. If $a$ is a minimal element of $\P$, then there exist linear extensions $U_a^-$ and $U_a^+$ of respectively $\U_a^-$ and $\U_a^+$, and linear extensions $U_{a,1}$ and $U_{a,2}$ of $\U_a$ such that the three following linear orders form a realizer of $\P$:
    \begin{gather*}
        U_a^{-} \; a \; U_a^{+}, \quad a \; U_{a,1}, \quad\text{and}\quad a \; U_{a,2}.
    \end{gather*}
    
    If $a$ is a maximal element of $\P$, then there exist linear extensions $D_a^-$ and $D_a^+$ of respectively $\D_a^-$ and $\D_a^+$, and linear extensions $D_{a,1}$ and $D_{a,2}$ of $\D_a$ such that the three following linear orders form a realizer of $\P$:
    \begin{gather*}
        D_a^{-} \; a \; D_a^{+}, \quad  D_{a,1} \; a, \quad\text{and}\quad D_{a,2} \; a.
    \end{gather*}
\end{proposition}

\begin{proof}
    We make a proof by induction on $n$ the number of elements of $\P$, our induction hypothesis containing both cases together.
    For $n=1$ it is true for both cases ($a$ is a minimal and maximal element of $\P$) using the three linear extensions $a$, $a$, and $a$.

    Assume that the results are true for any tree posets with $n$ or fewer elements. Suppose that $\P$ is a tree poset having $n+1$ elements.
    Without loss of generality we can assume that $a$ is a minimal element (otherwise we consider the dual poset and reverse the linear extensions, see Lemma \ref{lem:realizerPop}).
    Denote by $e_1,e_2,\dots , e_r$ the elements of $\U_a^{+}$ taken in prefix left-to-right order, and by $e_{s_1},e_{s_2},\dots , e_{s_r}$ these elements taken in prefix right-to-left order (looking at the tree $\U_a^+\cup \{a\}$ but without writing the root $a$).
    By Lemma \ref{lemma1}, a realizer of $\U_a^{+}$ is given by $e_1 e_2\cdots  e_r$ and $e_{s_1} e_{s_2}\cdots e_{s_r}$.
    See Figure \ref{fig:proofPropTree} for an illustration of the following. 
    
    For all $e_i\in \U_a^{+}$, let $\T_i$ be the biggest tree such that $\T_i\cap \U_a^{+}=\{e_i\}$.
    Then $\U_a^-=\bigsqcup_{i=1}^r \big( \T_i\setminus \{e_i\} \big)$.
    The trees $\T_i$ are all disjoint tree posets and contain $e_i$ as a maximal element. The notations $\D_{e_i},\,\D_{e_i}^-,\,\D_{e_i}^+$ will be used for the tree $\T_i$. In particular, $\D_{e_i}=\T_i \setminus \{e_i\}$.
    Since $|\T_i|\leq n$ because $a\not\in \T_i$, by the induction hypothesis we have a realizer of $\T_i$ for all $i\in [r]$ of the form
    \begin{gather*}
        D_i^{-} \; e_i \; D_i^{+}, \quad  D_{i,1} \; e_i, \quad D_{i,2} \; e_i,
    \end{gather*}
    where $D_i^{-}$ is a linear extension of $\D_{e_i}^{-}$, $D_i^{+}$ is a linear extension of $\D_{e_i}^{+}$, and $D_{i,1}$ and $D_{i,2}$ are linear extensions of $\D_{e_i}=\T_i \setminus \{e_i\}$.
    Let us consider the three following linear orders of $\P$:
    \begin{align}
       \prod_{i=r}^{1} \big(D_{i,1}\big)  \; a \; \prod_{i=1}^{r} e_i ,\label{eq:1}\\
        a \; \prod_{i=1}^{r} \big(D_{i,2} \; e_i \big),\label{eq:2}\\
        a \; \prod_{i=1}^{r} \big(D_{s_i}^{-} \; e_{s_i} \big) \; \prod_{i=r}^{1} D_{i}^{+}.\label{eq:3}
    \end{align}
    
    Let us prove that \eqref{eq:1}, \eqref{eq:2} and \eqref{eq:3} form a realizer of $\P$. The fact that these are three linear extensions is a straightforward verification.
    We recover the trees $\T_i$ for all $i\in [r]$ since $D_i^{-} \; e_i \; D_i^{+}$ is a subword of \eqref{eq:3}, we have $D_{i,1} \; e_i$ is a subword of \eqref{eq:1} and $D_{i,2} \; e_i$ is a subword of \eqref{eq:2}.
    Moreover, we recover the poset $\U_a^+$ since $e_1\cdots e_r$ is a subword of \eqref{eq:1} and \eqref{eq:2} and $e_{s_1}\cdots e_{s_r}$ is a subword of \eqref{eq:3}. 
    In order to be a realizer of $\P$, we need the following incompatibilities:
    $$\T_i\setminus \{e_i\}\not\sim a, \;\forall i,\qquad \quad \T_i\setminus \{e_i\}\not\sim \T_j\setminus\{e_j\},\;\forall i\neq j,$$
    $$\D_{e_i}^+\not\sim e_j, \;\forall i\neq j, \qquad \quad \D_{e_i}^- \not\sim \T_j,\; \text{ for } e_i\not< e_j.$$\\
    By looking at \eqref{eq:1} and \eqref{eq:2} we obtain the incomparabilities $\T_i\setminus \{e_i\}\not\sim a$ for all $i\in [r]$, and also $\T_i\setminus \{e_i\}\not\sim\T_j\setminus \{e_j\}$ for all $i\neq j$.
    By looking at \eqref{eq:1} and \eqref{eq:3}, we obtain $\D_{e_i}^{+} \not\sim e_j$ for all $i\neq j$.
    The remaining incomparabilities $\D_{e_i}^- \not\sim \T_j$ for $e_j\not < e_j$ follow from \eqref{eq:2} and \eqref{eq:3}.

    Finally, these three linear extensions are of the form we wanted; $\prod_{i=r}^{1} D_{i,1}$ is a linear extension of $\U_a^{-}$, $e_1 \cdots e_r$ is a linear extension of $\U_a^{+}$, and $\prod_{i=1}^{r} D_{i,2} \,e_i $ and $\prod_{i=1}^{r} \big(D_{s_i}^{-} \, e_{s_i} \big) \; \prod_{i=r}^{1} D_{i}^{+}$ are linear extensions of $\U_a$.
\end{proof}

The following result implies the well-known result that all tree posets are of dimension at most $3$.
This is needed in order to extract a particular realizer of a tree poset.

\begin{corollary}\label{cor:extlinarbre}
    Let $\P$ be a tree poset and $a\in \P$.
    The three following linear orders form a realizer of $\P$:
        \begin{gather}
            U_a^{-} \; D_{a,1} \; a \; U_a^{+}, \label{treea1} \\
            D_a^{-} \; a \; U_{a,1} \; D_a^{+}, \label{treea2} \\
            D_{a,2} \; a \; U_{a,2} \label{treea3},
        \end{gather}
    where $U_a^-$, $U_a^+$, $U_{a,1}$, and $U_{a,2}$ are the linear extensions obtained in Proposition \ref{prop:Hardcore} for the tree $\U_a \cup \{a\}$, and $D_a^-$, $D_a^+$, $D_{a,1}$, and $D_{a,2}$ are the linear extensions obtained in Proposition \ref{prop:Hardcore} for the tree $\D_a \cup \{a\}$.
\end{corollary}

\begin{proof}
    The fact that these three linear orders are linear extensions is straightforward (see Figure \ref{fig:simpletree}).
    Since these three linear extensions contain respectively the subwords $U_a^-\,a\,U_a^+$, $a\,U_{a,1}$ and $a\,U_{a,2}$, we know by Proposition \ref{prop:Hardcore} that their restrictions to the elements of $\U_a \cup \{a\}$ form a realizer of $\U_a \cup \{a\}$.
    The same goes for $\D_a \cup \{a\}$.
    It only remains to verify that $\D_a^+ \not\sim \U_a \cup \{a\}$, which follows from \eqref{treea2} and \eqref{treea3}, and $\U_a^- \not\sim \D_a\cup \{a\}$, which follows from \eqref{treea1} and \eqref{treea3}.
\end{proof}

\hspace{-.1cm}
\begin{minipage}{0.68\textwidth}
\begin{example} \label{ex:treeslinear}
For the tree on the right, the root is $a$ and, 
$\D_a^+$ is coloured orange, $\D_a^-$ red, $\U_a^-$ blue and $\U_a^+$ green.
Using Corollary \ref{cor:extlinarbre}, we obtain the following realizer:
\begin{align*}
\textcolor{blue}{f\ \ %14\ \ 15\ \ 16\ \ 
e\ \ g}\ \ \textcolor{red}{w}\ \ \textcolor{orange}{s}\ \ \textcolor{red}{v \ \ u}\ \ \textcolor{orange}{r}\ \ a\ \ \textcolor{olive}{i\ \ j},\\
\textcolor{red}{w\ \ v\ \ u}\ \ a\ \ \textcolor{olive}{i}\ \ \textcolor{blue}{e\ \ f\ \ %14\ \ 15\ \ 16\ \ 
g}\ \ \textcolor{olive}{j}\ \ \textcolor{orange}{r\ \ s}, \\
\textcolor{red}{v\ \ w}\ \ \textcolor{orange}{s}\ \ \textcolor{red}{u}\ \ \textcolor{orange}{r}\ \ a\ \ \textcolor{blue}{e}\ \ \textcolor{olive}{j\ \ i}\ \ \textcolor{blue}{f\ \ g%\ \ 14\ \ 16\ \ 15
}.
\end{align*}
\end{example}
\end{minipage}
\begin{minipage}{0.25\textwidth}
\begin{tikzpicture}[scale=0.8]
\filldraw[fill=green!20] (-0.3,0.3)--(-1.5,3)--(1.5,3)--(0.3,0.3)--(-0.3,0.3);
\filldraw[fill=blue!20] (0.4,0.5)--(3,0.5)--(1.5,3)--(0.4,0.5);
\filldraw[fill=red!20] (-0.3,-0.3)--(-1.75,-3.5)--(1.75,-3.5)--(0.3,-0.3)--(-0.3,-0.3);
\filldraw[fill=orange!20] (-0.395,-0.5)--(-3,-0.5)--(-1.75,-3.5)--(-0.395,-0.5);

\draw[thick] (.75,-3) -- (0,-2);
\draw[thick] (-.75,-3) -- (0,-2);
\draw[thick] (-.75,-3) -- (-1.85,-1.85);
\draw[thick] (0,-2) -- (-1.15,-1);
\draw[thick] (0,-2) -- (0,0);
\draw[thick] (0,0) -- (-0.6,2.5);
\draw[thick] (0,0) -- (0.65,2.5);
\draw[thick] (0.65,2.5) -- (1.15,0.95);
\draw[thick] (1.15,0.95) -- (1.45,1.95);
\draw[thick] (1.45,1.95) -- (1.95,0.95);

\draw (0,0) node(0) {$\bullet$};
\draw (0.15,0) node[right] {$a$};

\draw (.75,-3) node(-3) [red] {$\bullet$};
\draw (.75,-3) node[right] [red] {$v$};
\draw (-.75,-3) node(-2) [red] {$\bullet$};
\draw (-.75,-3) node[right] [red] {$w$};
\draw (0,-2) node(-1) [red] {$\bullet$};
\draw (0,-2) node[right] [red] {$u$};
\draw (-1.85,-1.85) node(-12) [orange] {$\bullet$};
\draw (-1.85,-1.85) node[left] [orange] {$s$};
\draw (-1.15,-1) node(-11) [orange] {$\bullet$};
\draw (-1.15,-1) node[left] [orange] {$r$};
\draw (1.15,0.95) node(11) [blue] {$\bullet$};
\draw (1.15,0.95) node[right] [blue] {$e$};
\draw (1.95,0.95) node(12) [blue] {$\bullet$};
\draw (1.95,0.95) node[right] [blue] {$f$};
\draw (-0.6,2.5) node(1) [olive] {$\bullet$};
\draw (-0.6,2.5) node[left] [olive] {$i$};
\draw (0.65,2.5) node(2) [olive] {$\bullet$};
\draw (0.65,2.5) node[left] [olive] {$j$};
\draw (1.45,1.95) node(13) [blue] {$\bullet$};
\draw (1.45,1.95) node[right] [blue] {$g$};
\end{tikzpicture}
\end{minipage}

\vspace{.5cm}

Let $\mathbf{Q}$ be a tree poset with $x,z\in \mathbf{Q}$ such that $[x,z]=\{x\lessdot y_1\lessdot \cdots \lessdot y_k\lessdot z\}$ has at least two elements and such that $\D_x^+=\emptyset$, $x$ is only covered by $y_1$, $\U_z^-=\emptyset$ and $z$ covers only $y_k$ (see Figure \ref{fig:chaintree}). We now define some subtrees of $\mathbf{Q}$. Denote by $\T_x$ the tree formed by $x$ and the elements below it, and by $\T_z$ the tree formed by $z$ and the elements above it. 
%Thus $\T_x=\D^-_x\sqcup \{x\}$ and $\T_z=\U^+_z\sqcup \{z\}$.
Let $\T:=\bigsqcup_{j=1}^k \T_{y_j}$ where $\T_{y_j}$ is the tree grafted to $y_i$ as in Figure \ref{fig:chaintree}. Let $\U:= \bigsqcup_{j=1}^k \U_{y_j}$ where $\U_{y_j}$ is defined as before in the tree $\T_{y_j}$, and similarly for $\U^+,\U^-,\D,\D^-,\D^+, \I^+,\I^-$.

\begin{figure}
    \begin{subfigure}{.49\textwidth}
    \centering
    \begin{tikzpicture}
        \begin{scope}[scale=0.3]
        \draw (0,0) node{$\bullet$};
        \draw (0,15) node{$\bullet$};
        \draw (0.2,0) node[right]{$x$};
        \draw (-0.2,15) node[left]{$z$};
        \draw (0,0) -- (1,-2) -- (-1,-2) -- (0,0) -- (0,5);
        \draw[dashed] (0,5)--(0,10);
        \draw (0,10)--(0,15)--(-1,17)--(1,17)--(0,15); 
        \begin{scope}[scale=0.4, xscale=1.5, yshift=6cm, rotate=-45]
        \draw (0,0) node[below right]{$y_1$};
        \draw (0,0) node{$\bullet$};
        \draw (-3,-5) -- (3,5) -- (-3,5) -- (3,-5) -- (-3,-5);
        \draw (1.8,3) -- (4.5,3) -- (3,5);
        \draw (-1.8,-3) -- (-4.5,-3) -- (-3,-5);     
        \end{scope}
        \begin{scope}[scale=0.4, xscale=1.5, yshift=18cm, rotate=-45]
        \draw (0,0) node[below right]{$y_j$};
        \draw (0,0) node{$\bullet$};
        \draw (-3,-5) -- (3,5) -- (-3,5) -- (3,-5) -- (-3,-5);
        \draw (1.8,3) -- (4.5,3) -- (3,5);
        \draw (-1.8,-3) -- (-4.5,-3) -- (-3,-5);     
        \end{scope}
         \begin{scope}[scale=0.4, xscale=1.5, yshift=30cm, rotate=-45]
        \draw (0,0) node[below right]{$y_k$};
        \draw (0,0) node{$\bullet$};
        \draw (-3,-5) -- (3,5) -- (-3,5) -- (3,-5) -- (-3,-5);
        \draw (1.8,3) -- (4.5,3) -- (3,5);
        \draw (-1.8,-3) -- (-4.5,-3) -- (-3,-5);     
        \end{scope}
        \end{scope}  
    \end{tikzpicture}
    \caption{The tree poset $\mathbf{Q}$ of Corollary \ref{cor:extlinarbrexy}.\label{fig:chaintree}}
    \end{subfigure}
    ~
    \begin{subfigure}{.49\textwidth}
    \centering
    \begin{tikzpicture}
        \begin{scope}[scale=0.4]
        \draw (0,2.2) -- (0,5);
        \draw[dashed] (0,5)--(0,10);
        \draw (0,10)--(0,11.8);
        \begin{scope}[scale=0.4, xscale=1.5, yshift=6cm, rotate=-45]
        \draw (0,0) node[below right]{$y^p_1$};
        \draw (0,0) node{$\bullet$};
        \draw (-3,-5) -- (3,5) -- (-3,5) -- (3,-5) -- (-3,-5);
        \draw (1.8,3) -- (4.5,3) -- (3,5);
        \draw (-1.8,-3) -- (-4.5,-3) -- (-3,-5);     
        \end{scope}
        \begin{scope}[scale=0.4, xscale=1.5, yshift=18cm, rotate=-45]
        \draw (0,0) node[below right]{$y^p_j$};
        \draw (0,0) node{$\bullet$};
        \draw (-3,-5) -- (3,5) -- (-3,5) -- (3,-5) -- (-3,-5);
        \draw (1.8,3) -- (4.5,3) -- (3,5);
        \draw (-1.8,-3) -- (-4.5,-3) -- (-3,-5);     
        \end{scope}
         \begin{scope}[scale=0.4, xscale=1.5, yshift=30cm, rotate=-45]
        \draw (0,0) node[below right]{$y^p_{k_p}$};
        \draw (0,0) node{$\bullet$};
        \draw (-3,-5) -- (3,5) -- (-3,5) -- (3,-5) -- (-3,-5);
        \draw (1.8,3) -- (4.5,3) -- (3,5);
        \draw (-1.8,-3) -- (-4.5,-3) -- (-3,-5);     
        \end{scope}
        \end{scope}  
    \end{tikzpicture}
    \caption{The tree $\T_p$.}\label{Tp}
    \end{subfigure}
    \caption{}
\end{figure}

\begin{corollary}\label{cor:extlinarbrexy}
    Let $\mathbf{Q}$ be as above and keep the notations from Corollary \ref{cor:extlinarbre}. The three following linear orders form a realizer of $\mathbf{Q}$:
    \begin{gather}
        \prod_{j=k}^{1} \left( U_{y_j}^{-} \; D_{y_j,1} \right)  \; T_{x,1} \; x \; \prod_{j=1}^{k} \left(y_j \; U_{y_j}^{+}\right) \; z \; T_{z,1}, \label{extlinarbre1xy}\\
        T_{x,2} \; x \; \prod_{j=1}^{k} \left(D_{y_j}^{-} \; y_j \right) \; z \; T_{z,2} \; \prod_{j=k}^{1} \left(U_{y_j,1} \; D_{y_j}^{+}\right) ,\label{extlinarbre2xy}\\
        T_{x,3} \; x \; \prod_{j=1}^k \left( D_{y_j,2} \; y_j\; U_{y_j,2}\right) \; z \; T_{z,3},\label{extlinarbre3xy}
    \end{gather}
    where $T_{x,1}\,x$, $T_{x,2}\,x$, $T_{x,3}\,x$ is any realizer of $\T_x$, and $z\,T_{z,1}$, $z\,T_{z,2}$, $z\,T_{z,3}$ is any realizer of $\T_z$.
\end{corollary}

\begin{proof}
    It is straightforward to check that these three linear orders are linear extensions. 
    For all $j\in [k]$, the realizer formed by $U_{y_{j}}^{-} \, D_{y_{j},1} \, {y_{j}} \, U_{y_{j}}^{+}$, and $D_{y_{j}}^{-} \, {y_{j}} \, U_{y_{j},1} \, D_{y_{j}}^{+}$ and $D_{y_{j},2} \, {y_{j}} \, U_{y_{j},2}$ of Corollary \ref{cor:extlinarbre} for $\T_{y_{j}}$ is obtained as subwords of \eqref{extlinarbre1xy}, \eqref{extlinarbre2xy} and \eqref{extlinarbre3xy}.
    Thus we know that for all $j\in [k]$ we recover the subposet $\T_{y_j}$.
    It is also clear that we recover the subposets $\T_x$ and $\T_z$.
    It remains to prove the incomparabilities:
    \begin{align}
        &\T_x \not\sim \D \sqcup \U^-,\qquad
        \T_z\not\sim \U \sqcup \D^+,\qquad \D^+_{y_i}, \U^-_{y_i}\not\sim y_j \text{ for all } i\neq j,\label{treeincomp1}\\
        &\U_{y_i}\not\sim y_j  \text{ for all }  i<j, \qquad \U_{y_i}\not\sim \U_{y_j} \sqcup \D^+_{y_j} \text{ for all } i\neq j,\label{treeincomp2}\\
        &\D^-_{y_i}\not\sim \U^+_{y_j}\text{ and } \D_{y_i}\not\sim y_j \text{ for all } i> j,\qquad   \D_{y_i}\not\sim \D_{y_j} \sqcup \U^-_{y_j} \text{ for all } i\neq j .\label{treeincomp3}
    \end{align}
    To verify the incomparabilities of \eqref{treeincomp1} it suffices to look at \eqref{extlinarbre1xy} and \eqref{extlinarbre2xy}. For \eqref{treeincomp2}, it suffices to look at \eqref{extlinarbre2xy} and \eqref{extlinarbre3xy}. Finally, for the incomparabilities of \eqref{treeincomp3} it suffices to look at \eqref{extlinarbre1xy} and \eqref{extlinarbre3xy}.
\end{proof}

\section{Non-extremal grafting on cycle posets}
\label{sectionTreesonCovers}

By non-extremal we mean the vertices that are neither minimal nor maximal elements.
Assume $n\geq 2$ and consider a cycle poset $\CC_n$ such that $x_{\lfloor \frac{p}{2}\rfloor +1}\lessdot y^p_1\lessdot \cdots \lessdot y^p_{k_p} \lessdot z_{\lceil \frac{p}{2}\rceil}$ is a saturated chain, for every $p\in [2n]$. These $p$ correspond to the edge labelling defined in Definition \ref{def:crown} and represented in Figure \ref{fig:crown} (see also Figure \ref{fig:runningexunicycle}). We denote by $ \TT(\CC_n)$ a family of rooted trees indexed by the elements of $\CC_n$ such that
$\T_{x_i}=\{r_{x_i}\}$ and $\T_{z_i}=\{r_{z_{i}}\}$, for every $i\in [n]$.
Thus, the grafting of $ \TT(\CC_n)$ on $\CC_n$, which defines a unicycle poset $\P_n$, only grafts nontrivial trees on the non-extremal elements of $\CC_n$ (see Figure \ref{fig:covercrown}). We denote by $\T_p:=\bigsqcup_{j=1}^{k_p} \T_{y^p_j}$ the tree obtained by grafting $\{(\T_{y^p_j},\,y^p_j)\mid~1\leq j\leq k_p\}$ on the saturated chain $y^p_1\lessdot \cdots \lessdot y^p_{k_p}$ of $\CC_n$ for every $p\in [2n]$ (see Figure \ref{Tp}). See the paragraph before Corollary \ref{cor:extlinarbrexy}, where we defined the $\U,\D,\I$ related to this tree.
For $n=1$, we only have two such trees: $\T_1$ on the maximal chain containing $a$, and $\T_2$ on the one containing $b$. This gives us a unicycle poset $\P_1$ (see Figure \ref{fig:coversunicycleC1}).
Using these notations, for any $p\in [2n]$ and any word $X$ denote:

\begin{align*}
    A_p(X) &:=\; \prod_{j=k_p}^{1} \left(U_{y_{j}^p}^{-} \; D_{{y_{j}^p},1} \right) \; X \; \prod_{j=1}^{k_p} \left(y^p_j \; U_{y^p_j}^+\right),\\
    B_p(Z) &:=\; \prod_{j=1}^{k_p} \left(D_{y^p_j}^- \; y^p_{j}\right) \; Z \; \prod_{j=k_p}^{1} \left( U_{y^p_j,1} \; D_{y^p_j}^+ \right) ,\\
    C_p &:=\; \prod_{j=1}^{k_p} \left(D_{y^p_j,2} \; y^p_j\; U_{y^p_j,2}\right) .
\end{align*}

\begin{remark} \label{remarkpartsAB}
In the last section of this paper, we will need the following observation. We have that $A_p(X)$ is composed of three parts; it starts with a linear extension of $\I^-_p$, then $X$, and then a linear extension of $\{y_1^p,\dots,y_{k_p}^p\} \sqcup \U_p^+$. Similarly for $B_p(Z)$, the first part is a linear extension of $\{y_1^p,\dots,y_{k_p}^p\} \sqcup \D_p^-$, then $Z$, and then a linear extension of $\I^+_p$.
\end{remark}

Recall that we always take the indices related to the crowns cyclically. We can rewrite the equations of Corollary \ref{cor:extlinarbrexy} where $x=x_{\lfloor \frac{p}{2}\rfloor +1}$, $z= z_{\lceil \frac{p}{2}\rceil}$ and the $y_j$'s replaced by the $y_j^{p}$'s, respectively for $p=2i-1$ and $p=2i$, in the following way where the realizer on the left if one for $\T_{2i-1}\sqcup \{x_i,z_i\}$, and on the right one for $\T_{2i}\sqcup \{x_{i+1},z_i\}$:
\begin{align*}
&A_{2i-1}(x_i)\, z_i, && A_{2i}(x_{i+1})\, z_i,\\
&x_i \; B_{2i-1}(z_i), && x_{i+1} \; B_{2i}(z_i),\\
&x_i \; C_{2i-1} \; z_i, && x_{i+1} \; C_{2i} \; z_i.
\end{align*}

Let us set $A_j^i(X):= A_i(A_j(X))$ and $B_j^i(X):= B_i(B_j(X))$.

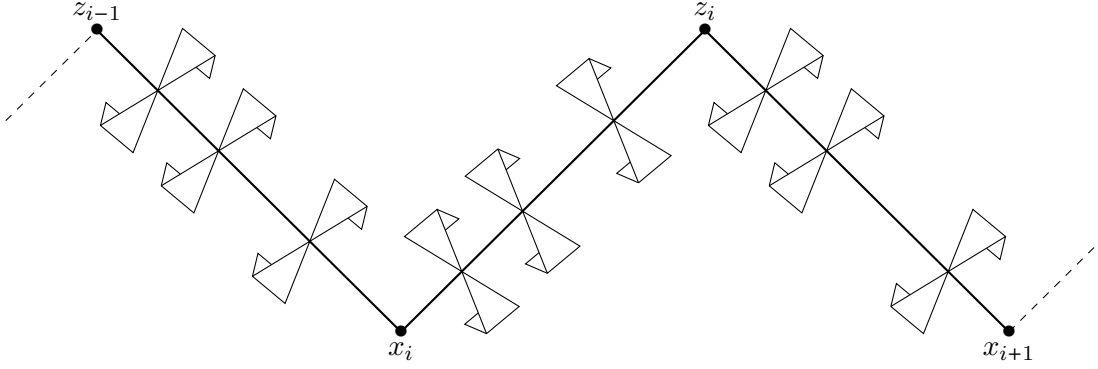
\begin{figure}
    \centering
    \begin{tikzpicture}
        \begin{scope}[scale=3]
            \draw[thick] (0,1)--(1,0)--(2,1)--(3,0);
            \draw[dashed] (-0.3,0.7)--(0,1);
            \draw[dashed] (3,0)--(3.3,0.3);
            \draw (0,1) node[above]{$z_{i-1}$};
            \draw (1,0) node[below]{$x_i$};
            \draw (2,1) node[above]{$z_{i}$};
            \draw (3,0) node[below]{$x_{i+1}$};
            \draw (0,1) node{$\bullet$};
            \draw (1,0) node{$\bullet$};
            \draw (2,1) node{$\bullet$};
            \draw (3,0) node{$\bullet$};
            \begin{scope}[xshift=0.2cm, yshift=0.8cm, scale=0.07, rotate=-40]
                \draw (-1,-3)--(1,3)--(-1,3)--(1,-3)--(-1,-3);
                \draw (1,3)--(1.5,2)--(0.67,2);
                \draw (-1,-3)--(-1.5,-2)--(-0.67,-2);
            \end{scope}
            
            \begin{scope}[xshift=0.4cm, yshift=0.6cm, scale=0.07, rotate=-40]
                \draw (-1,-3)--(1,3)--(-1,3)--(1,-3)--(-1,-3);
                \draw (1,3)--(1.5,2)--(0.67,2);
                \draw (-1,-3)--(-1.5,-2)--(-0.67,-2);
            \end{scope}
            
            \begin{scope}[xshift=0.7cm, yshift=0.3cm, scale=0.07, rotate=-40]
                \draw (-1,-3)--(1,3)--(-1,3)--(1,-3)--(-1,-3);
                \draw (1,3)--(1.5,2)--(0.67,2);
                \draw (-1,-3)--(-1.5,-2)--(-0.67,-2);
            \end{scope}
            
            \begin{scope}[xshift=1.2cm, yshift=0.2cm, scale=0.07, rotate=40]
                \draw (-1,-3)--(1,3)--(-1,3)--(1,-3)--(-1,-3);
                \draw (1,3)--(1.5,2)--(0.67,2);
                \draw (-1,-3)--(-1.5,-2)--(-0.67,-2);
            \end{scope}
            
            \begin{scope}[xshift=1.4cm, yshift=0.4cm, scale=0.07, rotate=40]
                \draw (-1,-3)--(1,3)--(-1,3)--(1,-3)--(-1,-3);
                \draw (1,3)--(1.5,2)--(0.67,2);
                \draw (-1,-3)--(-1.5,-2)--(-0.67,-2);
            \end{scope}
            
            \begin{scope}[xshift=1.7cm, yshift=0.7cm, scale=0.07, rotate=40]
                \draw (-1,-3)--(1,3)--(-1,3)--(1,-3)--(-1,-3);
                \draw (1,3)--(1.5,2)--(0.67,2);
                \draw (-1,-3)--(-1.5,-2)--(-0.67,-2);
            \end{scope}
            
            \begin{scope}[xshift=2.2cm, yshift=0.8cm, scale=0.07, rotate=-40]
                \draw (-1,-3)--(1,3)--(-1,3)--(1,-3)--(-1,-3);
                \draw (1,3)--(1.5,2)--(0.67,2);
                \draw (-1,-3)--(-1.5,-2)--(-0.67,-2);
            \end{scope}
            
            \begin{scope}[xshift=2.4cm, yshift=0.6cm, scale=0.07, rotate=-40]
                \draw (-1,-3)--(1,3)--(-1,3)--(1,-3)--(-1,-3);
                \draw (1,3)--(1.5,2)--(0.67,2);
                \draw (-1,-3)--(-1.5,-2)--(-0.67,-2);
            \end{scope}
            
            \begin{scope}[xshift=2.8cm, yshift=0.2cm, scale=0.07, rotate=-40]
                \draw (-1,-3)--(1,3)--(-1,3)--(1,-3)--(-1,-3);
                \draw (1,3)--(1.5,2)--(0.67,2);
                \draw (-1,-3)--(-1.5,-2)--(-0.67,-2);
            \end{scope}
            
        \end{scope}    
    \end{tikzpicture}
    \caption{A part of the crown illustrating Proposition \ref{horizCn} with added trees on its covers from left to right $\T_{2i-2}$, $\T_{2i-1}$ and $\T_{2i}$.}
    \label{fig:covercrown}
\end{figure}

\begin{proposition}\label{horizCn}
Let $n\geq 4$. The following three linear orders form a realizer of $\P_n$:
    \begin{gather}
        A_1(x_1) \; x_2\; B_2(z_1)\; C_{2n}\;\prod_{i=2}^{n-1}\Big(C_{2i-1}\; A_{2i}(x_{i+1})\; z_i \Big)\;B_{2n-1}(z_n), \label{linextABC1}\\
        A_3(x_2)\; \big(x_{3}\; B_4(z_2)\big)\; C_2\;\prod_{i=3}^{n-1}\Big(x_{i+1}\; B_{2i}^{2i-1}(z_i) \Big)\; C_{2n-1}\; A_{2n}(x_1)\; z_n\; B_1(z_1), \label{linextABC2}\\
        x_1\; A_{2n-1}(x_n) \; B_{2n}(z_n)\; C_1 \; \prod_{i=n-1}^3\Big(C_{2i}\; A_{2i-1}(x_i)\; z_i \Big)\; C_4\; A_2(x_2)\; B_3(z_2) \; z_1. \label{linextABC3}
    \end{gather}
\end{proposition}

\begin{proof}
    It is straightforward to check that these are three linear extensions of $\P_n$ (see Figure \ref{fig:covercrown}).
    By forgetting about the $A$, $B$ and $C$, meaning looking at the subwords on the elements $x_i$ and $z_i$, we recover using Lemma \ref{lem:ExtLinCouronne} the crown $\C_n$.
    Looking at the subwords of these linear extensions on the elements of $\T_{2i-1} \sqcup \{x_i,z_i\}$, we recover the realizer $A_{2i-1}(x_i)\, z_i$, $x_i \, B_{2i-1}(z_i)$ and $x_i \; C_{2i-1} \, z_i$ of this subposet, the same being true for $\T_{2i} \sqcup \{x_{i+1},z_i\}$ (see above Proposition \ref{horizCn}).
    The remaining incomparabilities to prove are the following:
    \begin{align}
        &\T_i \not\sim \T_j, \,\forall i\neq j,\label{incompABC1}\\
        &\T_{2i-1} \not\sim x_j,z_j, \,\forall j\neq i,\label{incompABC2}\\
        &\T_{2i} \not\sim x_j,z_k, \,\forall j\neq i+1,\; \forall k\neq i.\label{incompABC3}
    \end{align}
    
    Let $i\neq j$.
    All the elements of $\T_i$ are in $A_i(X)$, but also in $B_i(Z)$ and in $C_i$, regardless of $X$ and $Z$.
    Thus proving that $\T_i \not\sim \T_j$ is equivalent to showing that in a linear extension one of the symbol $A$, $B$ or $C$ with a subscript $i$ appears before one of those symbols with subscript $j$, which gives $\T_i<\T_j$, and, in another linear extension they appear in the opposite order, which gives $\T_j<\T_i$.
    Thus we just look at the sequences of subscripts of the symbols $A$, $B$ and $C$ obtained from \eqref{linextABC1}, \eqref{linextABC2} and \eqref{linextABC3}:
    \begin{gather*}
        1,2,2n,\prod_{i=2}^{n-1} (2i-1,2i),2n-1,\\
        3,4,2,\prod_{i=3}^{n-1} \substack{2i-1 \\ 2i},2n-1,2n,1,\\
        2n-1,2n,1,\prod_{i=n-1}^3 (2i,2i-1) ,4,2,3,
    \end{gather*}
    where, in the second linear extension, because of $B_{2i}^{2i-1}(z_i)$, we can conclude neither $\T_{2i-1}<\T_{2i}$ nor $\T_{2i}<\T_{2i-1}$.
    It is thus straightforward to verify the incomparabilities \eqref{incompABC1} for any pair $i\neq j$.
    
    For \eqref{incompABC2} and \eqref{incompABC3}, let $i\in[n]$.
    First remark that, any letter that comes before $x_i$ in \eqref{exten1}, comes before $\T_{2i-1}$ in \eqref{linextABC1}.
    Similarly, any letter that comes after $z_i$ in \eqref{exten1}, comes after $\T_{2i-1}$ in \eqref{linextABC1}.
    The same goes for the pairs \eqref{exten2}, \eqref{linextABC2} and, \eqref{exten3}, \eqref{linextABC3}.
    Because $x_i<z_i$, an element incomparable to both $x_i$ and $z_i$ appears before $x_i$ in one of the linear extensions and after $z_i$ in another.
    Therefore, for any $j,k\in[n]$, having both $x_i\not\sim x_j,z_k$ and $z_i\not\sim x_j,z_k$ imply $\T_{2i-1}\not\sim x_j,z_k$.
    We proceed similarly for $\T_{2i}$.
    Thus we have (see Figure \ref{fig:covercrown}):
    \begin{align}
        &\T_{2i-1} \not\sim x_j,z_k,\quad \text{ for all } j\neq i,i+1\text{ and }\, k\neq i,i-1,\\
        &\T_{2i} \not\sim x_j,z_k,\quad \text{ for all } j\neq i,i+1\text{ and }\, k\neq i,i+1.
    \end{align}
    
    To finish the proof, we need to look at the linear extension containing $C_{2i-1}$.
    In this linear extension we have $z_{i-1}<\T_{2i-1}<x_{i+1}$.
    Looking at the two others, one can easily verify that, at least one of the two has $x_{i+1}<\T_{2i-1}<z_{i-1}$ and \eqref{incompABC2} is satisfied.
    Similarly, in the linear extension containing $C_{2i}$, we have $z_{i+1}<\T_{2i}<x_{i-1}$.
    In at least one of the others, we have $x_{i-1}<\T_{2i}<z_{i+1}$ and \eqref{incompABC3} is satisfied. This finishes the proof.
\end{proof}

For $n=3$, it is exactly the same formula as in Proposition \ref{horizCn} but without the middle products of \eqref{linextABC2} and \eqref{linextABC3}. The proof is exactly the same but we state this particular case separately as we will need to treat $\P_3$ separately in the last section.

\begin{proposition}\label{horizC3}
The following three linear orders form a realizer of $\P_3$:
    \begin{gather}
        A_1(x_1) \; x_2 \; B_2(z_1) \; C_{6} \; C_{3}\; A_{4}(x_3)\; z_2 \; B_{5}(z_3), \label{C3linextABC1}\\
        A_3(x_2)\; x_{3}\; B_4(z_2) \; C_2\; C_{5}\; A_{6}(x_1)\; z_3\; B_1(z_1),\label{C3linextABC2}\\
        x_1\; A_{5}(x_3) \; B_{6}(z_3)\; C_1 \; C_4\; A_2(x_2)\; B_3(z_2) \; z_1. \label{C3linextABC3}
    \end{gather}
\end{proposition}

\begin{proposition}\label{horizC2}
The following three linear orders form a realizer of $\P_2$:
    \begin{align}
        A_1(x_1)\; x_2\; B_{2}(z_1)\; B_4^3 (z_2),\label{C_2BEV1}\\
        A_3^2 (x_2)\; A_4(x_1)\; z_2\; B_1(z_1),\label{C_2BEV2}\\
        x_1\; C_4\; x_2\; C_1\; C_3\; z_2\; C_2\; z_1.\label{C_2BEV3}
    \end{align}
\end{proposition}

\begin{proof}
    It is straightforward to check that these are three linear extensions of $\P_2$.
    By forgetting about the $A$, $B$ and $C$, meaning looking at the subwords on the elements $x_i$ and $z_i$, we recover the crown $\C_2$.
    Looking at the subwords of these linear extensions on the elements of $\T_{2i-1} \sqcup \{x_i,z_i\}$, we recover the realizer $A_{2i-1}(x_i)\, z_i$, $x_i \, B_{2i-1}(z_i)$ and $x_i \; C_{2i-1} \, z_i$ of this subposet, the same being true for $\T_{2i} \sqcup \{x_{i+1},z_i\}$ (see above Proposition \ref{horizCn}).
    The remaining incomparabilities to prove are the following:
    \begin{align*}
         &\T_i \not\sim \T_j ,\, \forall i\neq j,\qquad \T_1\not\sim x_2,z_2,\qquad \T_3 \not\sim x_1,z_1,\\
        &\T_2\not\sim x_1,z_2,\qquad \T_4\not\sim x_2,z_1.\\
    \end{align*}
Let $i\neq j$. In a similar manner as in Proposition \ref{horizCn}, to prove that $\T_i\not\sim \T_j$, we look at the sequences of subscripts of the symbols $A$, $B$ and $C$ in \eqref{C_2BEV1}, \eqref{C_2BEV2} and \eqref{C_2BEV3}.
    This gives $1,2,\,\substack{3\\4}$, and $\substack{2\\3}\,,4,1$ and $4,1,3,2$.
    Then $\T_i\not\sim \T_j$ follows by the fact that in one of the sequence $i$ appears before $j$ and in another one $j$ appears before $i$.
    We have that $\T_1\not\sim x_2,z_2$ and $\T_3 \not\sim x_1,z_1$ follow from \eqref{C_2BEV1} and \eqref{C_2BEV2} just by looking at $A_1,B_1$ and $A_3,B_3$ respectively. For $\T_2\not\sim x_1,z_2$ we look at $A_2,C_2$ in \eqref{C_2BEV2} and \eqref{C_2BEV3}. For $\T_4\not\sim x_2,z_1$ we look at $B_4,C_4$ in \eqref{C_2BEV1} and \eqref{C_2BEV3}.
\end{proof}

\begin{figure}
    \centering
    \begin{tikzpicture}
        \begin{scope}[scale=2.5]
            \draw (0,0)--(1,1)--(0,2)--(-1,1)--(0,0); 
            \draw (0.1,0) node{$x$};
            \draw (0,0) node{$\bullet$};
            \draw (-1.1,1) node{$a$};
            \draw (-1,1) node{$\bullet$};
            \draw (0.1,0) node{$x$};
            \draw (0,0) node{$\bullet$};
            \draw (1.1,1) node{$b$};
            \draw (1,1) node{$\bullet$};
            \draw (0.1,0) node{$x$};
            \draw (0,0) node{$\bullet$};
            \draw (0.1,2) node{$z$};
            \draw (0,2) node{$\bullet$};

            \draw [decorate,decoration={brace,amplitude=7pt,mirror,raise=2ex}]
  (1.2,0.2) -- (1.2,1.8) node[midway,xshift=3em,yshift=0em]{$\T_2$};
  \draw [decorate,decoration={brace,amplitude=7pt,raise=2ex}]
  (-1.2,0.2) -- (-1.2,1.8) node[midway,xshift=-3em,yshift=0em]{$\T_1$};

            \begin{scope}[xshift=1cm,yshift=1cm,scale=0.1]
                \draw (-1,-3)--(1,3)--(-1,3)--(1,-3)--(-1,-3);
                \draw (1,3)--(1.5,2)--(0.67,2);
                \draw (-1,-3)--(-1.5,-2)--(-0.67,-2);
            \end{scope}

            \begin{scope}[xshift=-1cm,yshift=1cm,scale=0.1]
                \draw (-1,-3)--(1,3)--(-1,3)--(1,-3)--(-1,-3);
                \draw (1,3)--(1.5,2)--(0.67,2);
                \draw (-1,-3)--(-1.5,-2)--(-0.67,-2);
            \end{scope}
            
            \begin{scope}[xshift=0.4cm,yshift=0.4cm,scale=0.07,rotate=40]
                \draw (-1,-3)--(1,3)--(-1,3)--(1,-3)--(-1,-3);
                \draw (1,3)--(1.5,2)--(0.67,2);
                \draw (-1,-3)--(-1.5,-2)--(-0.67,-2);
            \end{scope}
            
            \begin{scope}[xshift=0.6cm,yshift=0.6cm,scale=0.07,rotate=40]
                \draw (-1,-3)--(1,3)--(-1,3)--(1,-3)--(-1,-3);
                \draw (1,3)--(1.5,2)--(0.67,2);
                \draw (-1,-3)--(-1.5,-2)--(-0.67,-2);
            \end{scope}
            
            \begin{scope}[xshift=0.4cm,yshift=1.6cm,scale=0.07,rotate=-40]
                \draw (-1,-3)--(1,3)--(-1,3)--(1,-3)--(-1,-3);
                \draw (1,3)--(1.5,2)--(0.67,2);
                \draw (-1,-3)--(-1.5,-2)--(-0.67,-2);
            \end{scope}
            
            \begin{scope}[xshift=0.6cm,yshift=1.4cm,scale=0.07,rotate=-40]
                \draw (-1,-3)--(1,3)--(-1,3)--(1,-3)--(-1,-3);
                \draw (1,3)--(1.5,2)--(0.67,2);
                \draw (-1,-3)--(-1.5,-2)--(-0.67,-2);
            \end{scope}
            
            \begin{scope}[xshift=-0.4cm,yshift=0.4cm,scale=0.07,rotate=-40]
                \draw (-1,-3)--(1,3)--(-1,3)--(1,-3)--(-1,-3);
                \draw (1,3)--(1.5,2)--(0.67,2);
                \draw (-1,-3)--(-1.5,-2)--(-0.67,-2);
            \end{scope}
            
            \begin{scope}[xshift=-0.6cm,yshift=0.6cm,scale=0.07,rotate=-40]
                \draw (-1,-3)--(1,3)--(-1,3)--(1,-3)--(-1,-3);
                \draw (1,3)--(1.5,2)--(0.67,2);
                \draw (-1,-3)--(-1.5,-2)--(-0.67,-2);
            \end{scope}
            
            \begin{scope}[xshift=-0.4cm,yshift=1.6cm,scale=0.07,rotate=40]
                \draw (-1,-3)--(1,3)--(-1,3)--(1,-3)--(-1,-3);
                \draw (1,3)--(1.5,2)--(0.67,2);
                \draw (-1,-3)--(-1.5,-2)--(-0.67,-2);
            \end{scope}
            
            \begin{scope}[xshift=-0.6cm,yshift=1.4cm,scale=0.07,rotate=40]
                \draw (-1,-3)--(1,3)--(-1,3)--(1,-3)--(-1,-3);
                \draw (1,3)--(1.5,2)--(0.67,2);
                \draw (-1,-3)--(-1.5,-2)--(-0.67,-2);
            \end{scope}
        \end{scope}
    \end{tikzpicture}
    \caption{A cycle poset $\CC_1$ with grafted trees from Proposition \ref{horizC1}.}
    \label{fig:coversunicycleC1}
\end{figure}
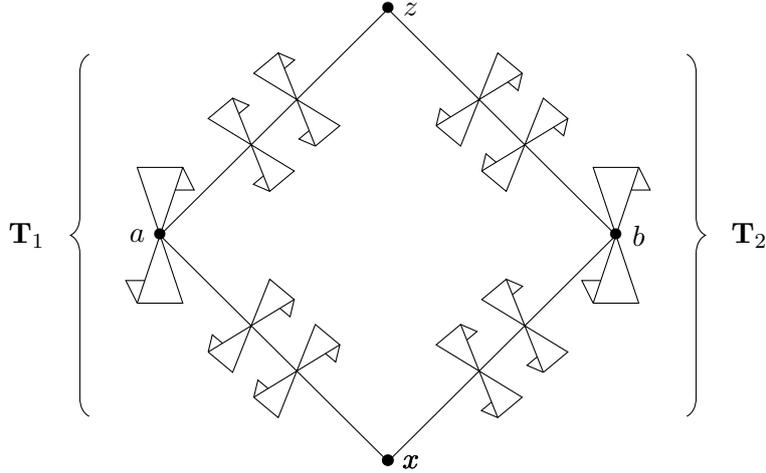

\begin{proposition}\label{horizC1}
The following three linear orders form a realizer of $\P_1$:
    \begin{align}
        A_1(x)\; B_{2}(z),\label{C_1BEV1}\\
        A_2(x)\; B_1(z),\label{C_1BEV2}\\
        x\; C_1\; C_2\;z.\label{C_1BEV3}
    \end{align}
\end{proposition}

\begin{proof}
    It is straightforward to check that these are three linear extensions of $\P_1$ (see Figure \ref{fig:coversunicycleC1}). Also we recover the crown $\C_1$ as a subposet since we have respectively in these linear extensions the subwords $x\,a\,b\,z$, $x\,b\,a\,z$ and $x\,a\,b\,z$.
    For $i=1,2$, looking at the subwords of these linear extensions on the elements of $\T_{i} \sqcup \{x,z\}$, we recover the realizer $A_{i}(x)\, z$, $x \, B_{i}(z)$ and $x \; C_{i} \, z$ of this subposet.
    The only incompatibility to check is $\T_1\not\sim \T_2$, which follows from $A_1 < B_2$ in \eqref{C_1BEV1} and $A_2<B_1$ in \eqref{C_1BEV2}.
\end{proof}

\section{Grafting on the crowns}
\label{sectionTreesonVertices}

\begin{figure}
    \centering
    \begin{tikzpicture}
        \begin{scope}[scale=1.5]
            \draw[thick] (0,0)--(1,1)--(2,0)--(3,1)--(4,0)--(5,1)--(6,0);   
            \draw (6,0) node[right]{$x_1$};
            \begin{scope}[scale=0.1, xscale=0.7, yscale=1.5]
                \draw (0,0) node[left]{$x_1$};
                \draw (-3,-5) -- (3,5) -- (-3,5) -- (3,-5) -- (-3,-5);
                \draw (1.8,3) -- (4.5,3) -- (3,5);
                \draw (-1.8,-3) -- (-4.5,-3) -- (-3,-5);    
            \end{scope}
            \begin{scope}[xshift=2cm]
                \begin{scope}[scale=0.1, xscale=0.7, yscale=1.5]
                \draw (0,0) node[left]{$x_2$};
                    \draw (-3,-5) -- (3,5) -- (-3,5) -- (3,-5) -- (-3,-5);
                   \draw (1.8,3) -- (4.5,3) -- (3,5);
                \draw (-1.8,-3) -- (-4.5,-3) -- (-3,-5);      
                \end{scope}
            \end{scope}
            \begin{scope}[xshift=4cm]
                \begin{scope}[scale=0.1, xscale=0.7, yscale=1.5]
                \draw (0,0) node[left]{$x_3$};
                    \draw (-3,-5) -- (3,5) -- (-3,5) -- (3,-5) -- (-3,-5);
                   \draw (1.8,3) -- (4.5,3) -- (3,5);
                \draw (-1.8,-3) -- (-4.5,-3) -- (-3,-5);     
                \end{scope}
            \end{scope}
            \begin{scope}[xshift=1cm, yshift=1cm]
                \begin{scope}[scale=0.1, xscale=0.7, yscale=1.5]
                \draw (0,0) node[left]{$z_1$};
                    \draw (-3,-5) -- (3,5) -- (-3,5) -- (3,-5) -- (-3,-5);
                    \draw (1.8,3) -- (4.5,3) -- (3,5);
                \draw (-1.8,-3) -- (-4.5,-3) -- (-3,-5);      
                \end{scope}
            \end{scope}
            \begin{scope}[xshift=3cm, yshift=1cm]
                \begin{scope}[scale=0.1, xscale=0.7, yscale=1.5]
                \draw (0,0) node[left]{$z_2$};
                    \draw (-3,-5) -- (3,5) -- (-3,5) -- (3,-5) -- (-3,-5);
                    \draw (1.8,3) -- (4.5,3) -- (3,5);
                \draw (-1.8,-3) -- (-4.5,-3) -- (-3,-5);     
                \end{scope}
            \end{scope}
            \begin{scope}[xshift=5cm, yshift=1cm]
                \begin{scope}[scale=0.1, xscale=0.7, yscale=1.5]
                \draw (0,0) node[left]{$z_3$};
                    \draw (-3,-5) -- (3,5) -- (-3,5) -- (3,-5) -- (-3,-5);
                    \draw (1.8,3) -- (4.5,3) -- (3,5);
                \draw (-1.8,-3) -- (-4.5,-3) -- (-3,-5);      
                \end{scope}
            \end{scope}
        \end{scope} 
    \end{tikzpicture}
    \caption{The crown $\C_3$ with grafted trees to each of its vertices. The last edge on the right is supposed to go back to the first vertex $x_1$.\label{fig:extremalcrown}}
\end{figure}
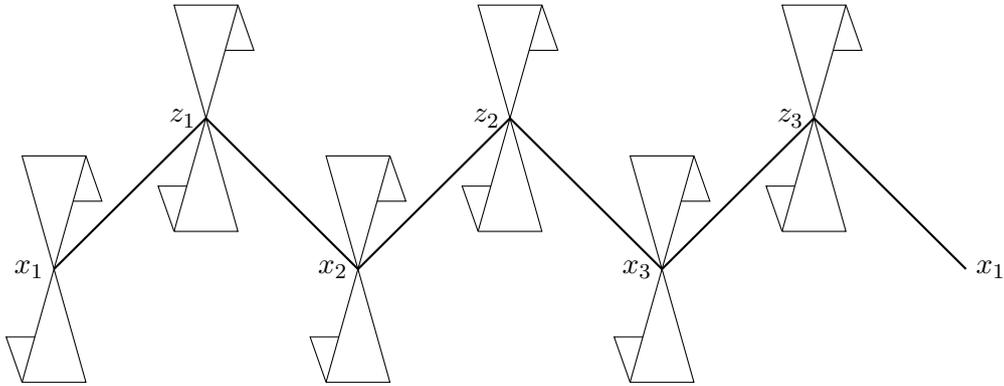

We will describe realizers for posets obtained by grafting trees on the crowns.
Recall that in Corollary \ref{cor:extlinarbre} we proved that the following is a realizer of $\T_w$:
$$(U_{w}^{-} \; D_{w,1}) \; (w \; U_{w}^{+}), \qquad
     (D_{w}^{-} \; w) \; (U_{w,1} \; D_{w}^{+}), \qquad
    D_{w,2} \; w \; U_{w,2}.$$
We denote $I_{w}^-:= U_{w}^{-} \, D_{w,1}$, $I_{w}^+:= U_{w,1} \, D_{w}^{+}$, $W^+:=w \, U_{w}^{+}$, $W^-:=D_{w}^{-} \, w$, and $W^{\bullet}~=~D_{w,2} \; w \; U_{w,2}$. Thus a realizer of $\T_w$ is given by:
\begin{gather}
    I_{w}^-\; W^+, \label{tree1} \\
    W^-\; I_{w}^+, \label{tree2} \\
    W^\bullet . \label{tree3}
\end{gather}

We will replace $w$ by $x_i$ or $z_i$, thus we just defined $I_{x_i}^-,I_{x_i}^+,\,X_i^-,\,X_i^+,\,X_i^{\bullet}$ and $I_{z_i}^-,I_{z_i}^+,\,Z_i^-,\,Z_i^+,\,Z_i^{\bullet}$.

For the remainder of this section, for any $n> 1$, let $\TT(\C_n)$ be a family of rooted tree posets indexed by the elements of $\C_n$, and let $\P_n$ be the unicycle poset obtained from the grafting of $\TT(\C_n)$ on $\C_n$ (see Figure \ref{fig:extremalcrown}).

\begin{proposition}\label{vertCn}
  Let $n\geq 4$. The following three linear orders form a realizer of $\P_n$:
    \begin{gather}
        I_{z_1}^-\; I_{x_2}^-\; X_1^-\; X_2^+\; Z_1^+\;
        (X_3^-\; Z_2^- \; I_{z_2}^+) \prod_{i=3}^{n-1}\Bigl(X_{i+1}^-\; Z_i^-\; I_{z_i}^+\; I_{x_i}^+ \Bigr)\;
        Z_n^\bullet \; I_{x_n}^+\; I_{x_1}^+,  \label{ext1}\\
        I_{z_2}^-\; X_2^\bullet\;
        \prod_{i=2}^{n-2}\Bigl(I_{z_{i+1}}^-\; I_{x_{i+1}}^-\; X_{i+1}^+\;Z_i^+ \Bigr)\; (I_{x_n}^-\; X_n^+\; Z_{n-1}^+)\;
        X_1^\bullet\; Z_n^-\;Z_1^\bullet\; I_{z_n}^+,   \label{ext2}\\
        I_{z_n}^-\; I_{x_1}^-\; X_1^+\;X_n^\bullet\;Z_n^+\; \prod_{i=n-1}^3\Bigl(X_i^\bullet\; Z_i^\bullet \Bigr)\; X_2^-\; Z_2^\bullet\; Z_1^-\; I_{z_1}^+\; I_{x_2}^+.   \label{ext3}
    \end{gather}
\end{proposition}
\begin{proof}
    It is straightforward to check that these three linear orders are linear extensions of $\P_n$ (see Figure \ref{fig:extremalcrown}).
    Moreover, the restrictions of the three linear extensions to the set $\T_w$ form a realizer of $\T_w$ for any $i\in [n]$ and $w\in \{x_i,z_i\}$ (see above Proposition \ref{vertCn}). The same goes for the crown $\C_n$, thus for $\C_n$ with $\X_i^-$ grafted to each $x_i$ and $\Z_i^+$ grafted to each $z_i$.
    We need to check the following incomparabilities for all $i\in[n]$:
    \begin{align}
    &\T_{x_i}\not\sim\T_{x_j}, \,\forall j\neq i ,  \label{(a)} \\
    &\T_{x_i}\not\sim\T_{z_j}, \,\forall j\neq i,i-1 , \label{(b)} \\
    &\I^+_{x_i}\not\sim\T_{z_j},\, \text{ for } j=i,i-1 , \label{(c)} \\
    &\T_{x_i}\not\sim\I^-_{z_j},\,\text{ for } j=i,i-1, \label{(d)} \\
    &\T_{z_i}\not\sim\T_{z_j},\,\forall j\neq i . \label{(e)}
    \end{align}

Note that for example in \eqref{ext1} the elements of $\T_{x_1}$ are in $X_1^-$ and $I^+_{x_1}$.
Keep in mind \eqref{tree1}, \eqref{tree2} and \eqref{tree3} to understand the following.
We first verify the incomparabilities \eqref{(a)} to \eqref{(d)} for each $i$ and finish with the incomparabilities of \eqref{(e)}.
We begin with $i=1$.
In \eqref{ext3}, $\T_{x_1}$ is smaller than everything except $\I_{z_n}^-$ and in \eqref{ext2}, it is bigger than everything except $\T_{z_1}$ and $\T_{z_n}$. 
This verifies that we have the incompatibilities of \eqref{(a)} and \eqref{(b)} and $\T_{x_i}\not\sim\I^-_{z_n}$.
In \eqref{ext1}, $\I^-_{z_1}<\T_{x_1}$, which completes the verification of \eqref{(d)}, and $\T_{z_1}, \T_{z_n}<\I^+_{x_1}$, which completes the verification of \eqref{(c)}.
    
For $i=2$, we already have all incomparabilities involving  $\T_{x_1}$, thus we restrict the linear extensions to the subset $\P\setminus\T_{x_1}$.
In \eqref{ext1}, $\T_{x_2}$ is smaller than every element except $\I_{z_1}^-$.
In \eqref{ext2}, $\T_{x_2}$ is smaller than every element except $\I_{z_2}^-$. In particular we have $\T_{x_2} \not\sim \I^-_{z_1},\I^-_{z_2}$, which gives \eqref{(d)}.
In \eqref{ext3}, $\T_{x_2}$ is bigger than everything except $\T_{z_1}$ and $\T_{z_2}$, thus verifying the incomparabilities \eqref{(a)} and \eqref{(b)}.
We also have $\T_{z_1}, \T_{z_2}<\I_{x_2}^+$, which proves \eqref{(c)}.

    For $i=n$, similarly, we can restrict to the subset $\P\setminus (\T_{x_1}\cup\T_{x_2})$.
    In \eqref{ext2}, $\T_{x_n}$ is bigger than everything except $\Z_{n-1}^+$ and $\T_{z_n}$.
    In \eqref{ext3}, $\T_{x_n}$ is smaller than everything except $\I_{z_n}^-$.
    Thus we have the incompatibilities of \eqref{(a)}, \eqref{(b)} and \eqref{(d)}.
    For \eqref{(c)}, we need to remark that $\I_{x_n}^+$ is bigger than $\T_{z_{n-1}},\T_{z_n}$ in \eqref{ext2}, and we have $\T_{x_n}<\T_{z_{n-1}},\T_{z_n}$ in \eqref{ext1}.

    For $3\leq i\leq n-1$, we can restrict to the subset $\P\setminus (\T_{x_1}\cup\T_{x_2}\cup\T_{x_n})$. We have in \eqref{ext2} and \eqref{ext3} respectively 
    \begin{align}
    &\T_{z_j},\T_{x_k}<\T_{x_i}<\T_{z_l},\T_{x_m},\quad \text{for }j<i-1,\, k<i,\, i<l,\,i< m, \\
    &\T_{z_j},\T_{x_k}<\T_{x_i}<\T_{z_l},\T_{x_m},\quad \text{for }i<j,\, i<k,\, l\leq i,\, m<i, \label{la2}
    \end{align}
    verifying the incompatibilities \eqref{(a)} and \eqref{(b)}.
    We finally need to remark that, in \eqref{ext1}, we have $\T_{z_i},\T_{z_{i-1}}<\I^+_{x_i}$ and, in \eqref{ext2}, we have $\I^-_{z_i},\I^-_{z_{i-1}}<\T_{x_i}$.
    Both paired with \eqref{la2} verify, respectively, \eqref{(c)} and \eqref{(d)}.
    
    The only missing incomparabilities are the ones of the type \eqref{(e)}, which are satisfied because for all $i$, in \eqref{ext1} we have $\T_{z_i}<\T_{z_{i+1}}$ and in \eqref{ext3} we have $\T_{z_{i+1}}<\T_{z_i}$.
\end{proof}

For the crown $\C_3$, it is the same formula stated in Proposition \ref{vertCn} but without the middle products of \eqref{ext1}, \eqref{ext2} and \eqref{ext3}. The proof is exactly the same but we state this particular case separately as we will need to treat $\C_3$ separately in the last section.

\begin{proposition}\label{vertC3}
The following three linear orders form a realizer of $\P_3$:
    \begin{gather}
        I_{z_1}^-\; I_{x_2}^-\; X_1^-\; X_2^+\; Z_1^+\; (X_3^-\; Z_2^- \; I_{z_2}^+) \; Z_3^\bullet \; I_{x_3}^+\; I_{x_1}^+,  \label{C3ext1}\\
        I_{z_2}^-\; X_2^\bullet\; (I_{x_3}^-\; X_3^+\; Z_{2}^+)\; X_1^\bullet\; Z_3^-\;Z_1^\bullet\; I_{z_3}^+,   \label{C3ext2}\\
        I_{z_3}^-\; I_{x_1}^-\; X_1^+ \; X_3^\bullet \; Z_3^+ \; X_2^-\; Z_2^\bullet\; Z_1^-\; I_{z_1}^+\; I_{x_2}^+.   \label{C3ext3}
    \end{gather}
\end{proposition}

\begin{proposition}\label{vertC2}
The following three linear orders form a realizer of $\P_2$:
    \begin{gather}
        X_1^-\; X_2^-\; Z_1^\bullet\; Z_2^\bullet\; I_{x_1}^+\; I_{x_2}^+, \label{C2ext1}\\
        I_{z_2}^-\; I_{z_1}^-\; X_2^\bullet\; X_1^\bullet\; Z_2^+\; Z_1^+, \label{C2ext2}\\
        I_{x_1}^-\; X_1^+\;I_{x_2}^- \; X_2^+\; Z_2^-\; I_{z_2}^+\; Z_1^-\; I_{z_1}^+. \label{C2ext3}
    \end{gather}
\end{proposition}
\begin{proof}
    It is straightforward to check that each linear order is a linear extension of $\P_2$.
    Moreover, the restrictions of the three linear extensions to the set $\T_w$ form a realizer of $\T_w$ for any $i\in [n]$ and $w\in \{x_i,z_i\}$ (see above Proposition \ref{vertCn}).
    The same goes for the crown $\C_2$, thus for $\C_2$ with $\X_i^-$ grafted to each $x_i$ and $\Z_i^+$ grafted to each $z_i$.
    We only need to check the following incomparabilities
    \begin{align*}
    &\T_{x_1}\not\sim \T_{x_2},\qquad \T_{z_1}\not\sim \T_{z_2},\qquad \I^+_{x_i}\not\sim \P\setminus \T_{x_i}\text{ for } i=1,2,\qquad \I^-_{z_i}\not\sim \P\setminus \T_{z_i}\text{ for } i=1,2.\\
    % &\I^+_{x_i}\not\sim \P\setminus \T_{x_i},\,\text{ for } i=1,2,\qquad \I^-_{z_i}\not\sim \P\setminus \T_{z_i},\,\text{ for } i=1,2.\\ 
    \end{align*}
    We have $\T_{x_1}<\T_{x_2}$ in \eqref{C2ext3} and $\T_{x_2}<\T_{x_1}$ in \eqref{C2ext2}. Similarly, $\T_{z_1}<\T_{z_2}$ in \eqref{C2ext1} and $\T_{z_2}<\T_{z_1}$ in \eqref{C2ext3}.
    Finally, we finish the proof of the proposition with:
    \begin{align*}
        &\T_{z_1}, \T_{z_2}<\I_{x_1}^+,\I_{x_2}^+ \quad \text{in \eqref{C2ext1}},\\
        &\I_{z_1}^-,\I_{z_2}^-<\T_{x_1}, \T_{x_2} \quad  \text{in \eqref{C2ext2}},\\
        &\T_{x_1}, \T_{x_2}<\T_{z_1},\T_{z_2} \quad  \text{in \eqref{C2ext3}}.
    \end{align*}
\end{proof}

\begin{proposition}\label{vertC1}
    Let $\TT=\{(\T_x,r_x),(\T_z,r_z),(\{r_a\},r_a),(\{r_b\},r_b)\}$ be a family of rooted tree posets indexed by the elements of $\C_1$, and $\P$ be the unicycle poset obtained from the vertex-grafting of $\TT$ on $\C_1$.
    Then the following three linear orders form a realizer of $\P$:
    \begin{align}
        X^-\; a\; b\; Z^\bullet\; I_x^+, \label{C1ext1}\\
        I_z^-\; X^\bullet\;b\; a\; Z^+, \label{C1ext2}\\
        I_x^-\; X^+\; a\; b\; Z^-\; I_z^+. \label{C1ext3}
    \end{align}
\end{proposition}
\begin{proof}
    It is straightforward to check that these linear orders are linear extensions of $\P$ with $\C_1$ as a subposet.
    The restrictions of the three linear extensions to the set $\T_w$ form a realizer of $\T_w$ for any $w=x,z$ (see above Proposition \ref{vertCn}).
    Moreover, by Corollary \ref{cor:extlinarbrexy}, we know that the restrictions of these linear extensions form a realizer for both $\X^-\sqcup\Z^+\sqcup \{a,x,z\}$ and $\X^-\sqcup\Z^+\sqcup \{b,x,z\}$.
    We only need to check that $\I^+_{x}\not\sim \P\setminus \T_{x}$ and $\I^-_{z}\not\sim \P\setminus \T_{z}$, which is easy to see.
\end{proof}

\section{Grafting on cycle posets}
\label{sectionFinale}

Joining the work done in the previous sections, we now build, for any unicycle poset $\P$, a realizer of $\P$ of size 3.
As proved in Section \ref{sectionUnicycle}, a unicycle poset $\P$ can be obtained by grafting rooted trees $\TT(\CC_n)$ on a cycle poset $\CC_n$.
Moreover, $\CC_n$ can be itself built by adding elements between two elements that form a cover of its underlying crown $\C_n$ (Lemma \ref{lemma:cycleiscrown}).
We merge the realizers described in Section \ref{sectionTreesonCovers} with the ones described in Section \ref{sectionTreesonVertices}. Note that in the next result, the first two linear orders are split in two lines. 

\begin{figure}
    \centering
     \begin{tikzpicture}
        \begin{scope}[scale=3.5]
            \draw[thick] (0,1)--(1,0)--(2,1)--(3,0);
            \draw[dashed] (-0.3,0.7)--(0,1);
            \draw[dashed] (3,0)--(3.3,0.3);
            \draw (-0.05,1) node[left]{$z_{i-1}$};
            \draw (1.05,0) node[right]{$x_i$};
            \draw (1.95,1) node[left]{$z_{i}$};
            \draw (3.05,0) node[right]{$x_{i+1}$};
            \draw (0,1) node{$\bullet$};
            \draw (1,0) node{$\bullet$};
            \draw (2,1) node{$\bullet$};
            \draw (3,0) node{$\bullet$};
            
            \begin{scope}[yshift=1cm, scale=0.07]
                \draw (-1,-3)--(1,3)--(-1,3)--(1,-3)--(-1,-3);
                \draw (1,3)--(1.5,2)--(0.67,2);
                \draw (-1,-3)--(-1.5,-2)--(-0.67,-2);
            \end{scope}
            
            \begin{scope}[xshift=2cm, yshift=1cm, scale=0.07]
                \draw (-1,-3)--(1,3)--(-1,3)--(1,-3)--(-1,-3);
                \draw (1,3)--(1.5,2)--(0.67,2);
                \draw (-1,-3)--(-1.5,-2)--(-0.67,-2);
            \end{scope}
            
            \begin{scope}[xshift=1cm, scale=0.07]
                \draw (-1,-3)--(1,3)--(-1,3)--(1,-3)--(-1,-3);
                \draw (1,3)--(1.5,2)--(0.67,2);
                \draw (-1,-3)--(-1.5,-2)--(-0.67,-2);
            \end{scope}
            
            \begin{scope}[xshift=3cm, scale=0.07]
                \draw (-1,-3)--(1,3)--(-1,3)--(1,-3)--(-1,-3);
                \draw (1,3)--(1.5,2)--(0.67,2);
                \draw (-1,-3)--(-1.5,-2)--(-0.67,-2);
            \end{scope}
            
            \begin{scope}[xshift=0.3cm, yshift=0.7cm, scale=0.07, rotate=-40]
                \draw (-1,-3)--(1,3)--(-1,3)--(1,-3)--(-1,-3);
                \draw (1,3)--(1.5,2)--(0.67,2);
                \draw (-1,-3)--(-1.5,-2)--(-0.67,-2);
            \end{scope}
            
            \begin{scope}[xshift=0.5cm, yshift=0.5cm, scale=0.07, rotate=-40]
                \draw (-1,-3)--(1,3)--(-1,3)--(1,-3)--(-1,-3);
                \draw (1,3)--(1.5,2)--(0.67,2);
                \draw (-1,-3)--(-1.5,-2)--(-0.67,-2);
            \end{scope}
            
            \begin{scope}[xshift=0.7cm, yshift=0.3cm, scale=0.07, rotate=-40]
                \draw (-1,-3)--(1,3)--(-1,3)--(1,-3)--(-1,-3);
                \draw (1,3)--(1.5,2)--(0.67,2);
                \draw (-1,-3)--(-1.5,-2)--(-0.67,-2);
            \end{scope}
            
            \begin{scope}[xshift=1.3cm, yshift=0.3cm, scale=0.07, rotate=40]
                \draw (-1,-3)--(1,3)--(-1,3)--(1,-3)--(-1,-3);
                \draw (1,3)--(1.5,2)--(0.67,2);
                \draw (-1,-3)--(-1.5,-2)--(-0.67,-2);
            \end{scope}
            
            \begin{scope}[xshift=1.5cm, yshift=0.5cm, scale=0.07, rotate=40]
                \draw (-1,-3)--(1,3)--(-1,3)--(1,-3)--(-1,-3);
                \draw (1,3)--(1.5,2)--(0.67,2);
                \draw (-1,-3)--(-1.5,-2)--(-0.67,-2);
            \end{scope}
            
            \begin{scope}[xshift=1.7cm, yshift=0.7cm, scale=0.07, rotate=40]
                \draw (-1,-3)--(1,3)--(-1,3)--(1,-3)--(-1,-3);
                \draw (1,3)--(1.5,2)--(0.67,2);
                \draw (-1,-3)--(-1.5,-2)--(-0.67,-2);
            \end{scope}
            
            \begin{scope}[xshift=2.3cm, yshift=0.7cm, scale=0.07, rotate=-40]
                \draw (-1,-3)--(1,3)--(-1,3)--(1,-3)--(-1,-3);
                \draw (1,3)--(1.5,2)--(0.67,2);
                \draw (-1,-3)--(-1.5,-2)--(-0.67,-2);
            \end{scope}
            
            \begin{scope}[xshift=2.5cm, yshift=0.5cm, scale=0.07, rotate=-40]
                \draw (-1,-3)--(1,3)--(-1,3)--(1,-3)--(-1,-3);
                \draw (1,3)--(1.5,2)--(0.67,2);
                \draw (-1,-3)--(-1.5,-2)--(-0.67,-2);
            \end{scope}
            
            \begin{scope}[xshift=2.7cm, yshift=0.3cm, scale=0.07, rotate=-40]
                \draw (-1,-3)--(1,3)--(-1,3)--(1,-3)--(-1,-3);
                \draw (1,3)--(1.5,2)--(0.67,2);
                \draw (-1,-3)--(-1.5,-2)--(-0.67,-2);
            \end{scope}
            
        \end{scope}    
    \end{tikzpicture}
    \caption{A part of a general unicycle poset whose underlying crown is $\C_n$ for $n\geq 2$.}
    \label{fig:generalunicycle}
\end{figure}
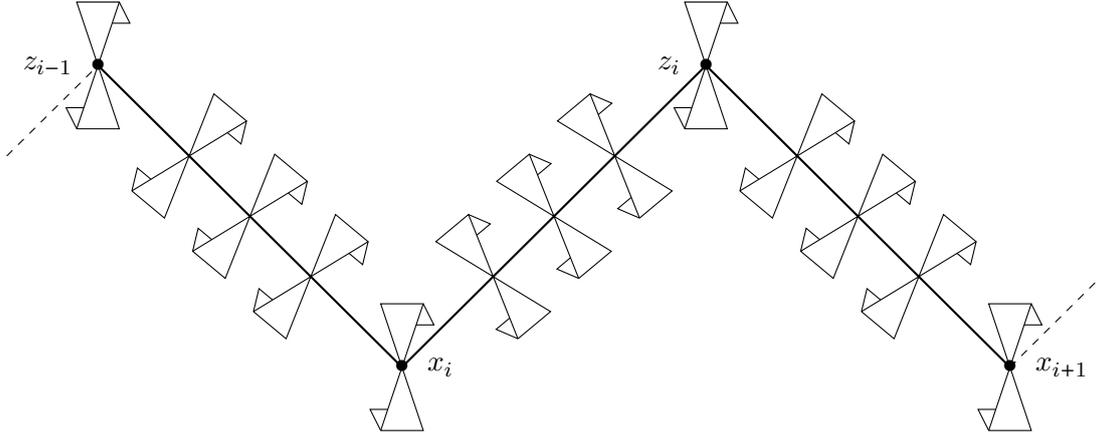

\begin{proposition}\label{conjFinal}
    Let $n\geq 4$. Let $\CC_n$ be a cycle poset, $\TT(\CC_n)$ be a family of rooted trees indexed by $\CC_n$, and $\P$ be the unicycle poset obtained from the grafting of $\TT(\CC_n)$ on $\CC_n$.
    The following three linear orders form a realizer of $\P$:
    {\small
    \begin{align}
        &I_{z_1}^-\; I_{x_2}^-\; A_1(X_1^-) \; X_2^+\; B_2(Z_1^+)\; C_{2n}\;
        \big(C_{3}\; A_4(X_3^-)\; Z_2^- \; I_{z_2}^+\big) \prod_{i=3}^{n-1}\Bigl(C_{2i-1}\; A_{2i}(X_{i+1}^-)\;Z_i^-\; I_{z_i}^+\; I_{x_i}^+ \Bigr)\label{compCnext1}\\
        \notag & \quad \quad \quad \quad \quad \quad \quad \quad \quad \quad \quad \quad \quad \quad \quad \quad \quad \quad \quad \quad\quad \quad \quad \quad \quad \quad \quad \quad \quad \quad \quad \quad \quad \quad \quad
        B_{2n-1}(Z_n^\bullet) \; I_{x_n}^+\; I_{x_1}^+,\\
        &I_{z_2}^-\; A_3(X_2^\bullet)\; \big(I_{z_3}^-\; I_{x_3}^-\; X_{3}^+\; B_4(Z_2^+)\big)\; C_2\;
        \prod_{i=3}^{n-2}\Bigl(I_{z_{i+1}}^-\; I_{x_{i+1}}^-\; X_{i+1}^+\; B_{2i}^{2i-1}(Z_i^+) \Bigr)\label{compCnext2}\\
        \notag & \quad \quad \quad \quad \quad \quad \quad \quad \quad \quad \quad \quad \quad \quad \quad \quad \quad \quad \quad \quad
        \big(I_{x_n}^-\; X_n^+\; B_{2n-2}^{2n-3}(Z_{n-1}^+)\big)\;
        C_{2n-1}\; A_{2n}(X_1^\bullet)\; Z_n^-\; B_1(Z_1^\bullet)\; I_{z_n}^+,\\
        &I_{z_n}^-\; I_{x_1}^-\; X_1^+\; A_{2n-1}(X_n^\bullet) \; B_{2n}(Z_n^+)\; C_1 \; \prod_{i=n-1}^3\Bigl(C_{2i}\; A_{2i-1}(X_i^\bullet)\; Z_i^\bullet \Bigr)\; C_4\; A_2(X_2^-)\; B_3(Z_2^\bullet) \; Z_1^-\; I_{z_1}^+\; I_{x_2}^+.\label{compCnext4}
    \end{align}
    }
\end{proposition}

\begin{proof}
    It is straightforward to check that these linear orders are linear extensions of $\P$ (see Figure \ref{fig:generalunicycle}). If we restrict these linear extensions respectively to the subposets considered in Propositions \ref{horizCn} and \ref{vertCn}, we recover the realizers of these subposets. Let $p\in [2n]$.
    It suffices to prove the following incomparabilities 
    \begin{align}
        \T_{p} &\not\sim \T_{x_j}, \T_{z_k}, \text{ for all } j\neq \lfloor p/2 \rfloor +1, k\neq \lceil p/2 \rceil, \label{fin1}\\
        \T_{p} &\not\sim \I^+_{x_{\lfloor p/2 \rfloor +1}},\I^-_{z_{\lceil p/2 \rceil}}, \label{fin2}\\
        \I_{p}^-&\not\sim \T_{x_{\lfloor p/2 \rfloor +1}},\label{fin3}\\
        \I_{p}^+&\not\sim \T_{z_{\lceil p/2 \rceil}}. \label{fin4}
    \end{align}
  
    Recall that $A_p$, $B_p$ and $C_p$ are linear extensions of $\T_p$. They appear in the linear extensions \eqref{compCnext1}, \eqref{compCnext2} and \eqref{compCnext4}, all in different ones.
    
    We designed our realizers such that $*$ in any $A_p(*)$ is either a $X_{\lfloor p/2\rfloor +1}^\bullet$ or $X_{\lfloor p/2\rfloor +1}^-$ but never $X_{\lfloor p/2\rfloor +1}^+$, and $*$ in any $B_p(*)$ is either a $Z_{\lceil p/2\rceil}^\bullet$ or $Z_{\lceil p/2\rceil}^+$ but never $Z_{\lceil p/2\rceil}^-$.
    It follows, by Remark \ref{remarkpartsAB}, that in the linear extension where $A_{p}$ appears, $\I^-_{p} <\T_{x_{\lfloor p/2 \rfloor +1}}$, and
    in the one where $B_{p}$ appears, $\T_{z_{\lceil p/2 \rceil}} <\I^+_{p}$. One can also note that in the linear extension where $C_{p}$ appears, $\T_{x_{\lfloor p/2 \rfloor +1}} <\T_{p}<\T_{z_{\lceil p/2 \rceil}}$.
    This proves \eqref{fin3} and \eqref{fin4}.

    We now prove \eqref{fin2}.
    In the linear extension where $X_{\lfloor p/2 \rfloor +1}^-$ appears, $\T_{p}<\I^+_{x_{\lfloor p/2 \rfloor +1}}$.
    Indeed, for $i\neq 2$, $X_i^-$ is in \eqref{compCnext1}, and either $i=1,n$ for which the assertion is trivial, or we have $\T_{2i-2}<\T_{2i-1}<\I_{x_i}^+$.
    Also $X_2^-$ is in \eqref{compCnext4} for which the assertion is trivial.
    In the one where $X_{\lfloor p/2 \rfloor +1}^+$ appears, $\T_{x_{\lfloor p/2 \rfloor +1}}<\T_{p}$.
    Indeed, for $i\neq 1,2$, $X_i^+$ is in \eqref{compCnext2} and $X_i^+<\T_{2i-2}<\T_{2i-1}$.
    We also have $X_1^+$ in \eqref{compCnext4} in which $\T_{x_1}<\T_{2n}<\T_1$, and $\X_2^+$ in \eqref{compCnext1} in which $\T_{x_2}<\T_2<\T_3$.
    This proves the part $\T_p\not\sim \I^+_{x_{\lceil p/2 \rceil}+1}$ of \eqref{fin2}.  
    Similarly, in the linear extension where $Z_{\lceil p/2 \rceil}^+$ appears, $\I^-_{z_{\lceil p/2 \rceil}}<\T_{p}$.
    Indeed, for $i\neq 1,n$, $Z_i^+$ is in \eqref{compCnext2}, and either $i=2$ for which the assertion is trivial, or we have $\I_{z_i}^-<\T_{2i},\T_{2i-1}$.
    We have that $Z_1^+$ and $Z_n^+$ appear respectively in \eqref{compCnext1} and \eqref{compCnext4} for which the assertion is trivial for both.
    In the one where $Z_{\lceil p/2 \rceil}^-$ appears, $\T_{p}<\T_{z_{\lceil p/2 \rceil}}$.
    Indeed, for $i\neq 1,n$, we have $Z_i^-$ is in \eqref{compCnext1} in which $\T_{2i-1}<\T_{2i}<Z_i^-$.
    We also have $Z_1^-$ in \eqref{compCnext4} in which $\T_{2n}<\T_{1}<\T_{z_1}$ and $Z_n^-$ in \eqref{compCnext1} in which $\T_{2n}<\T_{2n-1}<\T_{z_n}$.
    Thus proving the part $\T_p\not\sim \I^-_{z_{\lceil p/2 \rceil}}$ of \eqref{fin2}.

    It remains to prove \eqref{fin1}.
    
    For $p=1$, we want $\T_1\not\sim\T_{x_j},\T_{z_k}$ for $j,k\neq 1$.
    In \eqref{compCnext2} and \eqref{compCnext4}, we have respectively
    \begin{align*}
        &\T_{x_j},\T_{z_k}<\T_1, \quad &j\neq 1\text{ and } k\neq 1,n.\\      
        &\T_{x_n},\T_{z_n}<\T_1<\T_{x_j},\T_{z_k}, \quad &j\neq 1,n\text{ and } k\neq 1,n.
    \end{align*}
    The only incomparabilities that we still need to prove are  $\T_1\not\sim\T_{x_n},\T_{z_n}$, which follow from \eqref{compCnext1} where $\T_1<\T_{x_n},\T_{z_n}$.
    
    For $p=2$, we want $\T_2\not\sim\T_{x_j},\T_{z_k}$ for $j\neq 2$ and $k\neq 1$.
    In \eqref{compCnext2} and \eqref{compCnext4}, we have respectively
    \begin{align*}
        &\T_{x_3},\T_{z_2}<\T_2<\T_{x_j},\T_{z_k}, \quad &j\neq 2,3\text{ and } k\neq 1,3.  \\
        &\T_{x_j},\T_{z_k}<\T_2<\T_{z_2}, \quad &j\neq 2\text{ and } k\neq 1,2.
    \end{align*}
    The only incomparabilities that we still need to prove are $\T_2\not\sim\T_{x_3},\T_{z_3}$, which follow from \eqref{compCnext1} where $\T_2<\T_{x_3},\T_{z_3}$.

    For $p=2n$, we want $\T_{2n}\not\sim\T_{x_j},\T_{z_k}$ for $j\neq 1$ and $k\neq n$.
    In \eqref{compCnext1} and \eqref{compCnext2}, we have respectively
    \begin{align*}
        &\T_{x_2},\T_{z_1}<\T_{2n}<\T_{x_j},\T_{z_k},\quad &j\neq 1,2\text{ and } k\neq 1. \\
        &\T_{x_j},\T_{z_k}<\T_{2n}<\T_{z_1},\quad &j\neq 1\text{ and } k\neq 1,n.
    \end{align*}
    The only incomparability that we still need to prove is $\T_{2n}\not\sim\T_{x_2}$, which follows from \eqref{compCnext4} where $\T_{2n}<\T_{x_2}$.

    For $p=2n-1$, we want $\T_{2n-1}\not\sim\T_{x_j},\T_{z_k}$ for $j\neq n$ and $k\neq n$. In \eqref{compCnext2} and \eqref{compCnext4}, we have respectively
    \begin{align*}
        &\T_{x_j},\T_{z_k}<\T_{2n-1}<\T_{x_1},\T_{z_1},\quad &j\neq 1\text{ and } k\neq 1,n. \\
        &\T_{x_1}<\T_{2n-1}<\T_{x_j},\T_{z_k},\quad &j\neq 1,n\text{ and } k\neq n.
    \end{align*}
    The only incomparability that we still need to prove is $\T_{2n-1}\not\sim\T_{z_1}$, which follows from \eqref{compCnext1} where $\T_{z_1}<\T_{2n-1}$.
    
    For $p=2i$ with $2\leq i\leq n-1$, we want $\T_{2i} \not\sim \T_{x_j}, \T_{z_k}$ for $j\neq i+1$ and $k\neq i$. In \eqref{compCnext2} and \eqref{compCnext4}, we have respectively
    \begin{align*}
        &\T_{x_j}<\T_{2i}<\T_{x_1},\T_{x_l},\quad &1<j\leq i+1\text{ and } i+1<l. \\  
        &\T_{x_1},\T_{x_j},\T_{z_k}<\T_{2i}<\T_{x_l},\T_{z_m},\quad &i<j;\: i<k;\: l\leq i;\: m<i+1.
    \end{align*}
    The only incomparabilities that we still need to prove are $\T_{2i}\not\sim \T_{z_k}$ for $k\neq i$, which follow from \eqref{compCnext1} where
    $\T_{z_k}~<~\T_{2i}~<~\T_{z_m}$ for  $k<i$ and $i\leq m$.
    
    For $p=2i-1$ with $2\leq i\leq n-1$, we want $\T_{2i-1} \not\sim \T_{x_j}, \T_{z_k}$ for $j,k\neq i$. In \eqref{compCnext1}, \eqref{compCnext2} and \eqref{compCnext4}, we have respectively
    \begin{align*}
        &\T_{x_j},\T_{z_k}<\T_{2i-1}<\T_{x_l},\T_{z_m},\quad &j<i;\: k<i;\: i<l;\: i\leq m, \\
        &\T_{x_j}<\T_{2i-1}<\T_{x_1},\T_{x_l},\quad &1<j\leq i+1\text{ and } i+1<l, \\
        &\T_{z_k}~<~\T_{2i-1}~<~\T_{z_m},\quad &1<k<i\text{ and }m< i.
    \end{align*}

    Thus, we have $\T_{2i-1} \not\sim \T_{x_j}$, for all $j\neq i$, and
    $\T_{2i-1} \not\sim \T_{z_k}$, for all $k\neq i$.

    This proves \eqref{fin1}, which finishes the proof.
\end{proof}

The realizer for a unicycle poset that has $\CC_3$ as an underlying cycle poset is very similar. We will have the same linear extensions \eqref{compCnext1} and \eqref{compCnext4} but without the middle product, whereas \eqref{compCnext2} is slightly different; we are missing the middle product, as well as $I^-_{z_3}$ from the first parenthesis and we do not have the parenthesis after the product. But the proof is very similar to that of Proposition \ref{conjFinal}, thus we do not rewrite it.

\begin{proposition}
    Let $\CC_3$ be a cycle poset, $\TT(\CC_3)$ be a family of rooted trees indexed by $\CC_3$, and $\P$ be the unicycle poset obtained from the grafting of $\TT(\CC_3)$ on $\CC_3$.
    The following three linear orders form a realizer of $\P$:
    {\small
    \begin{align}
        &I_{z_1}^-\; I_{x_2}^-\; A_1(X_1^-) \; X_2^+\; B_2(Z_1^+)\; C_{6}\; C_{3}\; A_4(X_3^-)\; Z_2^- \; I_{z_2}^+ \; B_{5}(Z_3^\bullet) \; I_{x_3}^+\; I_{x_1}^+,\label{compC3ext1}\\
        &I_{z_2}^-\; A_3(X_2^\bullet)\; I_{x_3}^-\; X_{3}^+\; B_4(Z_2^+)\; C_2\;
        C_{5}\; A_{6}(X_1^\bullet)\; Z_3^-\; B_1(Z_1^\bullet)\; I_{z_3}^+,\label{compC3ext2}\\
        &I_{z_3}^-\; I_{x_1}^-\; X_1^+\; A_{5}(X_3^\bullet) \; B_{6}(Z_3^+)\; C_1 \; C_4\; A_2(X_2^-)\; B_3(Z_2^\bullet) \; Z_1^-\; I_{z_1}^+\; I_{x_2}^+.\label{compCn3ext3}
    \end{align}
    }
\end{proposition}

\begin{proposition}
    Let $\CC_2$ be a cycle poset, $\TT(\CC_2)$ be a family of rooted trees indexed by $\CC_2$, and $\P$ the unicycle poset obtained from the grafting of $\TT(\CC_2)$ on $\CC_2$.
    The following three linear orders form a realizer of $\P$:
    \begin{align}
        A_1(X_1^-)\; X_2^-\; B_{2}(Z_1^\bullet)\; B_4^3(Z_2^\bullet)\; I_{x_1}^+\; I_{x_2}^+,\label{compC2ext1}\\
        I_{z_2}^-\; I_{z_1}^-\; A_3^2 (X_2^\bullet)\; A_4(X_1^\bullet)\; Z_2^+\; B_1(Z_1^+),\label{compC2ext2}\\
        I_{x_1}^-\; X_1^+\; C_4\; I_{x_2}^-\; X_2^+\; C_1\; C_3\; Z_2^-\; I_{z_2}^+\; C_2\; Z_1^-\; I_{z_1}^+.\label{compC2ext3}
    \end{align}
\end{proposition}

\begin{proof}
    It is straightforward to check that these linear orders are linear extensions of $\P$. If we restrict these linear extensions respectively to the subposets considered in Propositions \ref{horizC2} and \ref{vertC2}, we recover the realizers of these subposets.
    Let $p\in [4]$. It suffices now to prove the following incomparabilities:
    \begin{align}
        &  \I_{p}^-\not\sim \T_{x_{\lfloor p/2 \rfloor +1}}, \label{plusmoins1}  \\
        &\I_{p}^+\not\sim \T_{z_{\lceil p/2 \rceil}}, \label{plusmoins2} \\
        &\T_p\not\sim \T_{z_{\lceil p/2 \rceil+1}},\I^+_{x_{\lfloor p/2 \rfloor+1}},\label{type1}\\
        &\T_p\not\sim \T_{x_{\lfloor p/2 \rfloor}},\I^+_{z_{\lceil p/2 \rceil}}.\label{type2}
    \end{align}
    
    Recall that $A_p$, $B_p$ and $C_p$ each appears only once in \eqref{compC2ext1}, \eqref{compC2ext2} and \eqref{compC2ext3}, all in different ones.
    One can note that, in the linear extension where $A_{p}$ appears, $\I^-_{p} <\T_{x_{\lfloor p/2 \rfloor +1}}$,
    in the one where $B_{p}$ appears, $\T_{z_{\lceil p/2 \rceil}} <\I^+_{p}$, and,
    in the one where $C_{p}$ appears, $\T_{x_{\lfloor p/2 \rfloor +1}} <\T_{p}<\T_{z_{\lceil p/2 \rceil}}$. This proves \eqref{plusmoins1} and \eqref{plusmoins2}.
    
    For \eqref{type1} and \eqref{type2}, in the case $p=1$, one only has to look at \eqref{compC2ext1} and \eqref{compC2ext2}.
    For $p\neq 1$, for the incomparabilities of \eqref{type1}, one has to look at \eqref{compC2ext1} and \eqref{compC2ext3}.
    For the incomparabilities of \eqref{type2}, one has to look at \eqref{compC2ext2} and \eqref{compC2ext3}.
\end{proof}

\begin{proposition}
    Let $\CC_1$ be a cycle poset, $\TT(\CC_1)$ be a family of rooted trees indexed by $\CC_1$, and $\P$ the unicycle poset obtained from the grafting of $\TT(\CC_1)$ on $\CC_1$.
    The following three linear orders form a realizer of $\P$:
    \begin{align}
        A_1(X^-)\; B_{2}(Z^\bullet)\; I_{x}^+, \label{compC1ext1}\\
        I_{z}^-\; A_2(X^\bullet)\; B_1(Z^+), \label{compC1ext2}\\
        I_{x}^-\; X^+\; C_1\; C_2\; Z^-\; I_{z}^+. \label{compC1ext3}
    \end{align}
\end{proposition}

\begin{proof}
    It is straightforward to check that these linear orders are linear extensions of $\P$ (see Figure \ref{fig:unicycleC1}).
    If we restrict these linear extensions respectively to the subposets considered in Propositions \ref{horizC1} and \ref{vertC1}, we recover the realizers of these subposets.
    It suffices now to prove the following incomparabilities
    \begin{align*}
        \I_x^+,\I_z^- &\not\sim \T_1,\T_2.
    \end{align*}
    With \eqref{compC1ext1} and \eqref{compC1ext3} we obtain respectively $\T_1,\T_2<\I_x^+$ and $\T_x<\T_1,\T_2$. Thus $\I_x^+\not\sim \T_1,\T_2$. Similarly for $\I_z^- \not\sim \T_1,\T_2$ with \eqref{compC1ext2} and \eqref{compC1ext3}.
\end{proof}

\section{Beyond the unicycle posets}

In this paper, we obtained specific realizers of the tree posets, that could maybe help answering the following questions.

\begin{question}
For a poset $\P$ of dimension $k\geq 3$, does grafting tree posets onto $\P$ changes the dimension?    
\end{question}

\begin{question}
% If $\P$ has exactly two cycles (which implies that the two cycles are disjoint) does its dimension is at most 3 ?
If $\P$ contains exactly two disjoint cycles, is its dimension at most 3?
\end{question}

More generally, we can ask:

\begin{question}
If the cycles of $\P$ are disjoint, is the dimension of $\P$ at most 3?    
\end{question}

\bibliographystyle{alpha}
\bibliography{sample}

\end{document}